\documentclass[11pt]{article}

\usepackage{latexsym}
\usepackage{amssymb}
\usepackage{amsthm}
\usepackage{amscd}
\usepackage{amsmath}
\usepackage{mathrsfs}
\usepackage{graphicx}
\usepackage{hyperref}
\usepackage{tabmac}
\usepackage{shuffle}

\usepackage[all]{xy}
\input xy \xyoption{frame}
\xyoption{dvips}

\usepackage[colorinlistoftodos]{todonotes}
\usepackage{fullpage}
\usepackage{ytableau}

\theoremstyle{definition}
\newtheorem* {theorem*}{Theorem}
\newtheorem* {conjecture*}{Conjecture}
\newtheorem{theorem}{Theorem}[section]

\theoremstyle{definition}

\newtheorem* {claim}{Claim}
\newtheorem* {example*}{Example}

\newtheorem{lemma}[theorem]{Lemma}
\theoremstyle{definition}
\newtheorem{definition}[theorem]{Definition}
\theoremstyle{definition}

\newtheorem{proposition}[theorem]{Proposition}
\newtheorem{corollary}[theorem]{Corollary}

\newtheorem{remark}[theorem]{Remark}
\theoremstyle{definition}
\newtheorem {example}[theorem]{Example}
\theoremstyle{definition}

\theoremstyle{definition}

\theoremstyle{definition}

\xyoption{dvips}

\def\({\left(}
\def\){\right)}

\newcommand{\CC}{\mathbb{C}}
\newcommand{\QQ}{\mathbb{Q}}
\newcommand{\cP}{\mathcal{P}}

\newcommand{\cR}{\mathcal{R}}

\newcommand{\cC}{\mathcal{C}}

\def\cX{\mathcal{X}}

\def\NN{\mathbb{N}}

\def\CC{\mathbb{C}}

\def\ZZ{\mathbb{Z}}

\def\GL{\mathrm{GL}}

\def\spanning{\textnormal{-span}}

\newcommand{\supp}{\mathrm{supp}}

\newcommand{\cL}{\mathcal{L}}

\def\fk{\mathfrak}

\def\barr{\begin{array}}
\def\earr{\end{array}}
\def\ba{\begin{aligned}}
\def\ea{\end{aligned}}
\def\be{\begin{equation}}
\def\ee{\end{equation}}

\def\Cyc{\mathrm{Cyc}}

\def\qquand{\qquad\text{and}\qquad}
\def\quand{\quad\text{and}\quad}
\def\qquord{\qquad\text{or}\qquad}

\def\inv{\mathrm{Inv}}

\def\I{\textsf{Invol}}

\def\DesR{\mathrm{Des}_R}
\def\DesL{\mathrm{Des}_L}

\def\omdef{\overset{\mathrm{def}}}

\def\id{\mathrm{id}}
\def\PP{\mathbb{P}}

\def\fkS{\fk S}

\def\ben{\begin{enumerate}}
\def\een{\end{enumerate}}

\def\fpf{{\tt {FPF}}}

\newcommand{\wfpf}{\Theta}
\def\DesF{\mathrm{Des}_R^\fpf}
\def\DesI{\mathrm{Des}_V}
\def\DesIF{\mathrm{Des}_{V}^{\fpf}}
\def\Dfpf{\hat D_{\fpf}}
\def\cfpf{\hat c_{\fpf}}

\def\flambda{\nu}

\def\Ffpf{\hat F^\fpf}
\def\Sfpf{\hat {\fk S}^\fpf}
\def\ifpf{\iota}

\newcommand{\xRightarrow}[2][]{\ext@arrow 0359\Rightarrowfill@{#1}{#2}}

\newcommand{\Fl}{\operatorname{Fl}}

\renewcommand{\O}{\operatorname{O}}
\newcommand{\Sp}{\operatorname{Sp}}

\newcommand{\Ess}{\operatorname{Ess}}

\newcommand{\cA}{\mathcal{A}}

\def\cAfpf{\cA_\fpf}
\def\cRfpf{\hat\cR_\fpf}

\def\F{\textsf{FPF}}
\newcommand{\arc}[2]{ \ar @/^#1pc/ @{-} [#2] }
\def\arcstop{\endxy\ }
\def\arcstart{\ \xy<0cm,-.15cm>\xymatrix@R=.1cm@C=.3cm }
\newcommand{\arcstartc}[1]{\ \xy<0cm,-.15cm>\xymatrix@R=.1cm@C=#1cm}

\def\Ffpf{\hat F^\fpf}
\def\Sfpf{\hat {\fk S}^\fpf}

\newcommand{\pf}{\operatorname{pf}}

\def\ellfpf{\hat\ell_\fpf}
\def\Pfpf{\hat\Psi}

\def\iT{\hat{\fk T}}
\def\fT{\iT^\fpf}

\def\fpsi{\eta_\fpf}

\newcommand{\minlex}{\operatorname{lt}}

\let\bxd\boxed
\renewcommand{\boxed}[1]{\hspace{0.7pt}\bxd{\!#1\!}\hspace{0.7pt}}

\def\Mfpf{\fk M^\fpf}

\def\Cfpf{\fk C^\fpf}
\def\Lfpf{L^\fpf}

\def\Fmap{\textsf{arc}}
\def\Imap{\textsf{dearc}}

\numberwithin{equation}{section}
\allowdisplaybreaks[1]
\UseCrayolaColors

\makeatletter
\renewcommand{\@makefnmark}{\mbox{\textsuperscript{}}}
\makeatother

\begin{document}
\title{Fixed-point-free involutions and Schur $P$-positivity}

\author{
    Zachary Hamaker \\
    Department of Mathematics \\
    University of Michigan \\
    { \tt zachary.hamaker@gmail.com}
 \and
    Eric Marberg \\
    Department of Mathematics \\
    HKUST \\
    {\tt eric.marberg@gmail.com}
\vspace{3mm}
\and
    Brendan Pawlowski\footnote{This author was partially supported by NSF grant 1148634.} \\
    Department of Mathematics \\
    University of Michigan \\
    {\tt br.pawlowski@gmail.com}
}

\date{}

\maketitle

\begin{abstract}
The orbits of the symplectic group acting on the type A flag variety are indexed by the fixed-point-free involutions in a finite symmetric group.
The cohomology classes of the closures of these orbits have polynomial representatives $\Sfpf_z$ akin to Schubert polynomials.
We show that the \emph{fixed-point-free involution Stanley symmetric functions} $\Ffpf_z$, which are stable limits of the polynomials $\Sfpf_z$, are Schur $P$-positive.
To do so, we construct an analogue of the Lascoux-Sch\"utzenberger tree, an algebraic recurrence that computes Schubert polynomials.
As a byproduct of our proof, we obtain a Pfaffian formula of geometric interest for $\Sfpf_z$ when $z$ is a fixed-point-free version of a Grassmannian permutation.
We also classify the fixed-point-free involution Stanley symmetric functions that are single Schur $P$-functions, and show that the 
decomposition of $\Ffpf_z$ into Schur $P$-functions is unitriangular with respect to dominance order on strict partitions.
These results and proofs mirror previous work by the authors related to the orthogonal group action on the type A flag variety.
 \end{abstract}

\setcounter{tocdepth}{2}
\tableofcontents

\section{Introduction}\label{intro-sect}

Fix a positive integer $n$ and
let $B \subset \GL_n(\CC)$ be the Borel subgroup of lower triangular matrices in the general linear group.
The orbits $\Omega_w$ of the opposite Borel subgroup of upper triangular matrices acting on 
the \emph{flag variety} $\Fl(n) = \GL_n(\CC)/B$ are indexed by permutations  $w\in S_n$ and their closures $X_w$ give $\Fl(n)$ a CW-complex structure.
The cohomology ring of $\Fl(n)$ has a presentation in terms of the \emph{Schubert polynomials} $\fkS_w$ introduced by Lascoux and Sch\"utzenberger \cite{LS}.
For the precise definition of $\fkS_w$, see Section~\ref{stab-sect}.

Schubert polynomials are of continued interest to both algebraic geometers and combinatorialists.
Computing the positive structure coefficients $c^w_{uv}$ in the expansion $\fkS_u \fkS_v = \sum c^w_{uv} \fkS_w$ 
remains a prominent open problem in algebraic combinatorics.
Among other interesting formulas,
 there is a generating function-type description of $\fkS_w$ in terms of the reduced words for $w$ \cite{BJS},
and a determinantal formula for $\fkS_w$ when $w$ is \emph{vexillary} ($2143$-avoiding) or \emph{fully commutative}
($321$-avoiding). When $w$ is \emph{dominant} ($132$-avoiding), $\fkS_w$ is a monomial.

Assume $n$ is even and consider the symplectic group $\Sp_n(\CC)$ acting on $\Fl(n)$.
There are again finitely many orbits, now indexed by the fixed-point-free involutions in $S_n$ \cite{RichSpring}.
For a fixed-point-free involution $z \in S_n$, the cohomology class of the corresponding orbit closure $Y_z$ is represented by the \emph{fixed-point-free involution Schubert polynomial} $\Sfpf_z$ introduced in~\cite{WY} and described precisely by Definition~\ref{fS-def}.
In~\cite{HMP1}, we gave a generating function-type description of $\Sfpf_z$ in terms of reduced words and derived a simple product formula for $\Sfpf_z$ when $z$ is a dominant fixed-point-free involution.
In this paper, we continue to study $\Sfpf_z$ and related combinatorics. Some of this combinatorics also appears in representation theory when studying the quasi-parabolic Iwahori-Hecke algebra modules  defined by Rains and Vazirani~\cite{RainsVazirani}.

The groups $\O_n(\CC)$ and $\GL_p(\CC) \times \GL_q(\CC)$ (with $p+q=n)$ also act on $\Fl(n)$ with finitely many orbits.
This paper is a continuation of the authors' previous work on the $\O_n(\CC)$ case~\cite{HMP4}.
The $\GL_p(\CC)\times \GL_q(\CC)$ case has not yet been as thoroughly investigated, though there has been some recent progress in \cite{burks2018reduced}; see also \cite{CJW,WYclans}.

The symmetric group $S_n$ of permutations of $[n] = \{1,2,\dots,n\}$ is a Coxeter group generated by the simple transpositions $s_i = (i,i+1)$ for $1 \leq i \leq n-1$.
For $u \in S_m$ and $v \in S_n$, we write $u \times v$ for the permutation in $S_{m+n}$ that maps $i \mapsto u(i)$ for $i \in [m]$ and  $m+i \mapsto m + v(i)$ for $i \in [n]$.
The \emph{Stanley symmetric function} of $w \in S_n$ is then the stable limit 
\[F_w \omdef = \lim_{m \to \infty} \fkS_{1_m \times w}\] where $1_m$ denotes the identity element of $S_m$.
This is a well-defined homogeneous symmetric function; see Section~\ref{stab-sect}.
These functions were introduced by Stanley to enumerate reduced words~\cite{Stan}.
Edelman and Greene showed bijectively that Stanley symmetric functions are Schur positive using an insertion algorithm~\cite{EG}.

A permutation is \emph{Grassmannian} if it has exactly one descent.
If $w \in S_n$ is Grassmannian then $\fkS_w$ is a Schur polynomial and $F_w$ is a Schur function~\cite[Proposition 2.6.8]{Manivel}.
One can show algebraically that $F_w$ is Schur positive by using the \emph{Lascoux-Sch\"utzenberger tree}~\cite{LS}, an iterated recurrence
for Schubert polynomials based on certain specializations of Monk's rule.
The Lascoux-Sch\"utzenberger tree decomposes $\fkS_w$ into a sum of Schubert polynomials indexed by Grassmannian permutations and other terms whose stable limits vanish.

Let $\F_n$ be the set of fixed-point-free involutions in $S_{2n}$. Define $\wfpf_n = (1,2)(3,4) \dots (2n{-1},2n) \in \F_n$.
The \emph{fixed-point-free involution Stanley symmetric function} of $z \in \F_n$ 
is the limit 
\[\Ffpf_z \omdef = \lim_{n \to \infty} \Sfpf_{\wfpf_n \times z}\]
which is a well-defined homogeneous symmetric function; see Section~\ref{fpf-schub-sect}.
We introduced these functions in \cite{HMP1}
to study the enumeration of certain analogues of reduced words.

The odd
power-sum functions $p_1, p_3, p_5,\dots$ generate a subalgebra $\Gamma$ of the usual algebra of symmetric functions $\Lambda$. This subalgebra has a distinguished basis $\{P_\lambda\}$ indexed by strict integer partitions,
whose elements $P_\lambda$ are the so-called \emph{Schur P-functions}.
See Section~\ref{ss:schur-p} for the precise definition.
In \cite{HMP1}
we conjectured the following statement, which is proved at the end of Section~\ref{fpf-pos-sect}:

\begin{theorem}\label{fpf-our-cor1}
Each $\Ffpf_z$ is \emph{Schur $P$-positive}, i.e., 
$\Ffpf_z \in \NN\spanning\{ P_\lambda : \text{$\lambda$ is a strict partition}\}$.
\end{theorem}

The first step in our proof of this result to identify the ``fixed-point-free'' analogue of a Grassmannian permutation
and then prove that $\Ffpf_z$ is a Schur $P$-function when $z$ is an involution of this type.
The precise definition of an \emph{FPF-Grassmannian} involution is sightly unintuitive; for the details, see Definition~\ref{d:fpf_grass}. 
We can easily describe which Schur $P$-function corresponds to an FPF-Grassmannian involution, however.

The \emph{(FPF-involution) code} of  $z \in \F_n$
is the sequence $\cfpf(z) = (c_1,c_2,\dots,c_{2n})$ in which $c_i$ is the number of positive integers $j$ 
with $j<i <z(j)$ and $j < z(i)$.
Define the \emph{shape} of $z \in \F_n$ to be the partition $\flambda(z)$ given by the transpose of the partition that sorts $\cfpf(z)$.
For example, if $z = 2n\cdots 321 = (1,2n)(2,2n-1)\cdots(n,n+1) \in \F_n$,
then $\cfpf(z) = (0,1,2,\dots,n-1,n-1,\dots,2,1,0)$ and $ \flambda(z) = (2n-2,2n-4,\dots,2).$
The following is proved as Theorem~\ref{fgrass-thm}.

\begin{theorem}
\label{t:fgrass-thm}
If $z \in \F_n$ is FPF-Grassmannian,
then $\flambda(z)$ is strict and $\Ffpf_z = P_{\flambda(z)}$.
\end{theorem}

The second step in our proof of Theorem~\ref{fpf-our-cor1}
is to define an analogue of the Lascoux-Sch\"utzenberger tree for fixed-point-free involutions.
We do this using the transition equations that we introduced in~\cite{HMP3}.
We show that 
repeated applications of these transition equations always result in a sum of $\Sfpf_z$'s
where $z$ is FPF-Grassmannian, along with other terms whose stable limits vanish.
The desired Schur $P$-positivity property follows from Theorem~\ref{t:fgrass-thm} on taking limits.


This proof can be recast as an algorithm to explicitly compute any $\Ffpf_z$.
By choosing an appropriate involution, one can use this algorithm to expand any product $P_\lambda P_\mu$ 
as a positive linear combination of Schur $P$-functions.
In this way, we obtain a new Littlewood-Richardson rule for Schur $P$-functions from our results (see Corollary~\ref{c:LR}).

It remains an open problem to find a bijective proof of Theorem~\ref{t:fgrass-thm}.
Since the FPF-transition equations have a bijective interpretation \cite{HMP3}, a bijective proof of Theorem~\ref{t:fgrass-thm} would,
in principle, lead to a bijective proof of Theorem~\ref{fpf-our-cor1}.
A more direct way of proving Theorem~\ref{fpf-our-cor1} bijectively would be to find an insertion algorithm for 
\emph{fixed-point-free involution words} (see Section~\ref{fpf-schub-sect}).

A permutation $w \in S_n$ is \emph{vexillary} if $F_w$ is a single Schur function.
Analogously, we say that $z \in \F_n$ is \emph{FPF-vexillary} if $\Ffpf_z$ is a single Schur $P$-function.
FPF-Grassmannian involutions are FPF-vexillary by Theorem~\ref{t:fgrass-thm}.
Stanley showed that $w \in S_n$ is vexillary if and only if $w$ avoids the pattern 2143.
A similar result holds for involutions; see Theorem~\ref{fpf-vex-thm} for the full statement.

\begin{theorem}
\label{intro-fvex-thm}
There is a pattern avoidance condition characterizing FPF-vexillary involutions.
\end{theorem}

The \emph{dominance order} on partitions is the partial order $\leq$ with $\lambda \leq  \mu$ 
if $\sum_{i=1}^{m} \lambda_i \leq \sum^m_{i=1} \mu_i$ for all $m \in \NN$.
In Section~\ref{fpf-tri-sect}, we show that the Schur $P$-expansion of $\Ffpf_z$ 
is unitriangular with respect to dominance order,
in the following sense:

\begin{theorem}\label{intro-fpf-tri-thm}
If $z \in \F_n$ then $\flambda(z)$ is strict and
$\Ffpf_z \in P_\flambda + \NN\spanning\{P_\lambda : \lambda  <\flambda(z)\}$.
\end{theorem}

We mention a quick application of these results.
The explicit version of Theorem~\ref{intro-fvex-thm} implies that 
the reverse permutation $2n\cdots 321 \in \F_n$ is FPF-vexillary. 
By Theorem~\ref{intro-fpf-tri-thm}, we therefore have
 $\Ffpf_{2n\cdots 321} = P_{\flambda(2n\cdots 321)} = P_{(2n-2,2n-4,\dots,2)}$.
In prior work, we proved that $\Ffpf_{2n\cdots 321} = (s_{\delta_n})^2$
where $s_\lambda$ is the Schur function of a partition $\lambda$
and $\delta_n=(n-1,\dots,3,2,1)$ \cite[Theorem 1.4]{HMP1}.
Combining these formulas shows that $ P_{(2n-2,2n-4,\dots,2)} = (s_{\delta_n})^2$,
which is a special case of \cite[Theorem V.3]{DeWitt}.

Assume $z \in \F_n$ is FPF-Grassmannian. 
The symmetric function  $\Ffpf_z = P_{\flambda(z)}$
can then be expressed as the Pfaffian of a matrix whose entries are Schur $P$-functions indexed by
partitions with at most two parts. This formula is essentially Schur's original definition of $P_\lambda$ in \cite{Schur}.
In general, the polynomial $\Sfpf_z$ is not equal to $P_{\flambda(z)}$ specialized to finitely many variables.
However, $\Sfpf_z$ has a similar Pfaffian formula which we sketch as follows.

There is an FPF-Grassmannian involution $z$ of shape $(n-\phi_1, n-\phi_2, \dots,n-\phi_r)$ 
associated to each sequence of integers $1 \leq \phi_1 < \phi_2< \dots < \phi_r \leq n$, and we define
 $\Sfpf[\phi_1,\phi_2,\dots,\phi_r;n] = \Sfpf_z$
to be the FPF-involution Schubert polynomial of this element. 
For the precise definition, see \eqref{sfpf-ffpf-eq}.
The following is  restated as Theorem~\ref{tt:pfaffian} and illustrated in a concrete case by Example~\ref{e:pfaffian}.

\begin{theorem}
\label{t:pfaffian}
Suppose  $1 \leq \phi_1 < \phi_2< \dots < \phi_r \leq n$ are integers.
Let $m$ be whichever of $r$ or $r+1$ is even.
Define $\fk M$ to be the $m\times m$ skew-symmetric matrix with
$\fk M_{ij} = - \fk M_{ji}= \Sfpf[\phi_i,\phi_j;n]$
whenever $ i<j$,
where $\Sfpf[\phi_i,\phi_{r+1};n] \omdef = \Sfpf[\phi_i;n]$.
Then
$
\Sfpf[\phi_1,\phi_2,\dots, \phi_r;n] = \pf \fk M
.$
\end{theorem}

Combining this identity with our Lascoux-Sch\"utzenberger tree for fixed-point-free involutions
gives an algorithm for expanding any $\Sfpf_z$ as a sum of Pfaffians.
One piece is missing to make this algorithm effective as a means of computing $\Sfpf_z$:
it remains an open problem to find a simple formula for the terms
$\Sfpf[\phi_i,\phi_j;n]$ appearing in the matrix $\fk M$ in Theorem~\ref{t:pfaffian}.
This is unexpectedly nontrivial.

There is a determinantal formula for $\fkS_w$ which holds when $w \in S_n$ is a vexillary permutation.
Analogously, there should exist a Pfaffian formula for $\Sfpf_z$ applicable when $z$ is any FPF-vexillary involution.
Such a formula would generalize Theorem~\ref{t:pfaffian} since 
FPF-Grassmannian involutions are FPF-vexillary.
There is also a determinantal formula for $\fkS_w$ when $w$ is fully commutative.
This formula should have an analogue for the polynomials $\Sfpf_z$; 
however, we do not yet know
what the appropriate ``fixed-point-free'' analogue of a fully commutative permutation should be.

Knutson, Lam, and Speyer have given a geometric interpretation of the Stanley symmetric function
$F_w$ as the representative for the class of a \emph{graph Schubert variety} in the Grassmannian $\mbox{Gr}(n,2n)$~\cite{KLS}.
It would be interesting to find a geometric interpretation of Theorem~\ref{fpf-our-cor1} in this vein.
Schur $P$-functions are cohomology representatives for Schubert varieties in the orthogonal Grassmannian. 
 We believe there is a way to adapt the construction of Knutson, Lam, and Speyer to give a subvariety of the orthogonal Grassmannian whose class is represented by $\Ffpf_z$, resulting in a geometric proof of Theorem~\ref{fpf-our-cor1}.
A similar approach should also relate $\O_n(\CC)$-orbit closures to the geometry of the Lagrangian Grassmannian.

\subsection*{Acknowledgements}

We thank
Dan Bump, Michael Joyce,
Vic Reiner,
Alex Woo,
Ben Wyser, and
Alex Yong for helpful conversations during the development of this paper.

\section{Preliminaries}
\label{prelim-sect}

Let $\PP\subset \NN \subset \ZZ$ denote the respective sets of positive, nonnegative, and all 
integers. For $n \in \PP$, let $[n]\omdef =\{1,2,\dots,n\}$.
The \emph{support} of a map $w : X\to X$ is the set  $\supp(w) \omdef= \{ i \in X: w(i)\neq i\}$.
Define $S_\ZZ$ as the group of permutations of $\ZZ$ with finite support,
and let $S_\infty\subset S_\ZZ$ be the subgroup of permutations with support contained in $ \PP$.
We view $S_n$ as the subgroup of permutations in $S_\infty$ fixing all integers outside $[n]$.

Throughout, we let $s_i\omdef =(i,i+1) \in S_\ZZ$ for $i \in \ZZ$.
Let $\cR(w)$ be the set of \emph{reduced words} for $w \in S_\ZZ$,
i.e., the  sequences $(s_{i_1},s_{i_2},\dots,s_{i_p})$
of simple transpositions of shortest possible length such that
$w = s_{i_1}s_{i_2} \dots s_{i_p}$.
Write $\ell(w)$ for the 
common length of each word in $\cR(w)$.
When $w : \ZZ \to \ZZ$ is any bijection,
we let $\DesR(w)$ (respectively, $\DesL(w)$) 
denote the set of simple transpositions $s_i$ for $i \in \ZZ$ with $w(i) > w(i+1)$
(respectively $w^{-1}(i) > w^{-1}(i+1)$).
If $w \in S_\ZZ$ then
 $\DesL(w) $ and $\DesR(w)$
 are the usual \emph{right} and \emph{left descent sets} of $w$, consisting of
 the simple transpositions $s$ such that $\ell(sw)<\ell(w)$ and
$\ell(ws)<\ell(w)$, respectively.

\subsection{Divided difference operators}\label{divided-sect}

We recall a few properties of \emph{divided difference operators}. Our main references are
\cite{Knutson, Manivel}.
Let $\cL \omdef= \ZZ\left[x_1,x_2,\dots,x_1^{-1},x_2^{-1},\dots\right]$ be the ring of Laurent
polynomials over $\ZZ$ in a countable set of commuting indeterminates, and let
$\cP \omdef= \ZZ[x_1,x_2,\dots]$ be the subring of polynomials in $\cL$.
The group $S_\infty$ acts on $\cL$ by permuting variables, and one defines
\be\label{partial-i-eq}
\partial_i f \omdef= (f - s_i f)/(x_i-x_{i+1})\qquad\text{for $i \in \PP$ and $f \in \cL$}.
\ee
The \emph{divided difference operator} $\partial_i$ defines a map $\cL \to \cL$ that restricts to
a map $\cP\to \cP$. 
It is clear by definition that $\partial_i f = 0$ if and only if $s_i f = f$.
If $f \in \cL$ is homogeneous and $\partial_i f \neq 0 $ then $\partial_i f $ is
homogeneous of degree $\deg(f)-1$.
If  $f,g \in \cL$  then $\partial_i(fg) =(\partial_if)g +  (s_if) \partial_i g$,
and if $\partial_i f = 0$, then
$\partial_i(fg) = f \partial_i g$.

For $i \in \PP$ the \emph{isobaric divided difference operator} $\pi_i : \cL \to \cL$ is defined by
\be\label{pi-i-eq}
\pi_i(f) \omdef= \partial_i(x_if) = f + x_{i+1}\partial_i f\qquad\text{for $f \in \cL$}.
\ee
Observe that $\pi_i f=f$ if and only if $s_i f=f$, in which case $\pi_i(fg) = f \pi_i(g)$ for
$g \in \cL$.
If
$f \in \cL$ is homogeneous  with   $\pi_i f \neq 0 $, then $\pi_i f $ is homogeneous of the same
degree. 
The operators $\partial_i$ and $\pi_i$ both satisfy the braid relations for $S_\infty$,
so we may define
$
\partial_w = \partial_{i_1}\partial_{i_2}\cdots \partial_{i_k}
$
and
$
\pi_w=\pi_{i_1}\pi_{i_2}\cdots \pi_{i_k}
$ for any $(s_{i_1},s_{i_2},\dots,s_{i_k}) \in \cR(w)$.
Moreover, one has $\partial_i^2=0$ and $\pi_i^2=\pi_i$ for all $i \in \PP$.

\subsection{Schubert polynomials and Stanley symmetric functions}\label{stab-sect}

Fix $n \in \PP$ and let $w_n\omdef =n\cdots321 \in S_n$ and $x^{\delta_n} \omdef = x_1^{n-1} x_2^{n-2}\cdots  x_{n-1}^1.$
The \emph{Schubert polynomial}  (see \cite{Knutson,Manivel})  of  $w \in S_n$ is the polynomial
\[\fkS_w \omdef= \partial_{w^{-1}w_n}x^{\delta_n} \in \cP.\]
This formula for $\fkS_w$ is independent of the choice of $n$
such that $w \in S_n$, and we consider
the Schubert polynomials to be a family indexed by $S_\infty$.
Since $\partial_i^2=0$, it follows that
\be\label{maineq}
\fkS_1=1
\qquand
\partial_i \fkS_w = \begin{cases} \fkS_{ws_i} &\text{if $s_i \in \DesR(w)$}
\\
0&
\text{if $s_i \notin \DesR(w)$}\end{cases}
\qquad\text{for each $i \in \PP$}.
\ee
Conversely,
one can show that
$\{ \fkS_w\}_{w \in S_\infty}$ is the unique family of homogeneous polynomials indexed by
$S_\infty$ satisfying \eqref{maineq}; see \cite[Theorem 2.3]{Knutson} or the introduction of \cite{BH}.
Each $ \fkS_w$ has degree $\ell(w)$,
and the polynomials $\fkS_w$ for $w \in S_\infty$ form a $\ZZ$-basis for $\cP$ 
\cite[Proposition 2.5.4]{Manivel}.

There is a useful formula for $\fkS_w$ as a sort of generating function over reduced words due to Billey, Jockusch, and Stanley \cite{BJS}.
Fix $w \in S_n$, and for each $a = (s_{a_1},s_{a_2},\dots,s_{a_k}) \in \cR(w)$, let $C(a)$ be the set of sequences  of positive integers
 $I = (i_1,i_2\dots,i_k) $ satisfying 
 \be\label{bjs-eq}
 i_1 \leq i_2 \leq \dots \leq i_k \qquand i_j < i_{j+1}\text{ whenever }a_j < a_{j+1}.\ee
 We write $I \leq a$ to indicate that  $i_j \leq a_j$ for all $j $ and define $x_I = x_{i_1} x_{i_2}\cdots x_{i_k}$.
The Schubert polynomial corresponding to $w \in S_n$ is then \cite[Theorem 1.1]{BJS}
\be\label{schub1-eq}
 \fkS_w = \sum_{a \in \cR(w)} \sum_{ \substack{ I \in C(a) \\ I \leq a} } x_I . \ee
For example, since $\cR(312) = \{(s_2,s_1)\}$ and $\cR(1342) = \{(s_2,s_3)\}$,
it holds that
\[ \fkS_{312}  = x_1^2 \qquand \fkS_{1342} =x_1x_2 + x_1x_3 + x_2x_3.\]
As expected, one has $\partial_1 \fkS_{312} = \partial_3 \fkS_{1342} = \fkS_{132} = x_1+x_2$.

Write 
 $\Lambda$ for the usual subring of
 bounded degree \emph{symmetric functions} in the ring of formal power series $\ZZ[[x_1,x_2,\dots]]$.
A sequence of power series $f_1,f_2,\dots $ has a limit $\lim_{n\to \infty} f_n\in \ZZ[[x_1,x_2,\dots]]$
  if the
 coefficient sequence of each fixed monomial is eventually constant.
For any map $w  : \ZZ \to \ZZ$ and $N\in \ZZ$, let $w \gg N : \ZZ \to \ZZ$ be the map
$i\mapsto w(i-N) +N$. 

\begin{definition}\label{F-def}
If $w \in S_\ZZ$ then the limit
\[
F_w \omdef= \lim_{N\to \infty} \fkS_{w\gg N}=
 \sum_{a \in \cR(w)} \sum_{ I \in C(a)} x_I  \in \ZZ[[x_1,x_2,\dots]]
 \]
  is the \emph{Stanley symmetric function} of $w$.
\end{definition}

The second equality in this definition follows from \eqref{schub1-eq}.
Stanley introduced these power series and proved that they are symmetric in \cite{Stan}. (The indexing conventions of \cite{Stan} differ from ours by the transformation of indices $w \mapsto w^{-1}$.)
The symmetric function $F_w$ is 
homogeneous
of degree $\ell(w)$, 
and the coefficient of any square-free monomial in $F_w$ is $|\cR(w)|$.
For example, 
\[ F_{321} = \sum_{i<j<k} 2x_ix_jx_k + \sum_{i<j} (x_i^2x_j + x_ix_j^2)\]
 and $|\cR(321)| = |\{ (s_1,s_2,s_1),(s_2,s_1,s_2)\}| = [x_1x_2x_3] F_{321} = 2$.


Definition~\ref{F-def} makes it clear
that $F_w = F_{w\gg N}$ for any $N \in \ZZ$,
but does not tell us how to efficiently compute these symmetric functions.
It is well-known result of Edelman and Greene \cite{EG} that each $F_w$ is Schur positive; for
a brief account of one way to compute the
corresponding Schur expansion, see \cite[\S4.2]{HMP4}.
We require one other definition of $F_w$.

\begin{lemma}[Macdonald \cite{Macdonald}] \label{stab-lem}
If $w \in S_\infty$ then $F_w = \lim_{n \to \infty} \pi_{w_n} \fkS_w$.
\end{lemma}

\begin{proof}
This is reproved in \cite[\S3]{HMP1}: the claim follows from \cite[Proposition 3.37 and Theorem 3.39]{HMP1}.
\end{proof}

\subsection{FPF-involution Schubert polynomials}\label{fpf-schub-sect}

For $n \in \PP$, let $\F_n$ be the set of
permutations 
$z \in S_n$ with $z=z^{-1}$ and $z(i)\neq i$ for all $i \in [n]$.
Let $\F_\infty$ and $\F_\ZZ$ be the $S_\infty$- and
$S_\ZZ$-conjugacy classes of the permutation $\wfpf:\ZZ \to \ZZ$ given by 
\be\label{wfpf-eq}
\wfpf : i \mapsto i - (-1)^i.
\ee
We refer to elements of $\F_n$, $\F_\infty$, and $\F_\ZZ$ as \emph{fixed-point-free (FPF) involutions}.
Note that
 $\F_n$ is empty if $n$ is odd.
For $z \in \F_\ZZ$ and $N \in \ZZ$, we see $z\gg N \in \F_\ZZ$ if and only if $N$ is even.
While technically $\F_n\not\subset \F_\infty$, there is a natural
inclusion
\be
\label{ifpf-def}
\ifpf: \F_n \hookrightarrow \F_\infty
\ee
given by the map
 that sends $z \in \F_n$ to the permutation of $\ZZ$ whose  restrictions to $[n]$ and to
 $\ZZ\setminus [n]$ coincide respectively with those of $z$ and $\wfpf$.
In symbols, we have 
$
\ifpf(z) = z \cdot \wfpf \cdot s_1 \cdot s_3 \cdot s_5\cdots s_{n-1}$.
Note we obtain $\wfpf_n = (1,2)(3,4) \dots (2n{-}1,2n)$ by restricting $\wfpf$ to $[2n]$.

We identify  elements  of $\F_n$, $\F_\infty$, or $\F_\ZZ$ with the complete
matchings on  $[n]$, $\PP$, or $\ZZ$ with
distinct vertices connected by an edge whenever they form a nontrivial cycle.
We depict such matchings with the vertices  on a horizontal axis, ordered from left to right,
and edges shown as convex curves in the upper half plane. For example,
\[
 (1,6)(2,7)(3,4)(5,8) \in \F_8
 \qquad
 \text{is represented as}
 \qquad
\arcstart{
*{.}  \arc{1.}{rrrrr}   & *{.}  \arc{1.}{rrrrr}   & *{.} \arc{.5}{r} & *{.} & *{.} \arc{.8}{rrr} & *{.} & *{.} & *{.}
}\endxy
\]
We will omit the numbers labeling the vertices in these matchings if they remain clear from context.

For each $z \in \F_\ZZ$, define
\be\label{inv-eq}
\ba
\inv(z) &= \{(i,j) \in \ZZ \times \ZZ: i < j, z(i) > z(j)\},
\\
 \Cyc_\ZZ(z) &= \{(i,j) \in \ZZ \times \ZZ: i < j = z(i)\},
 \ea
\ee
so that
$\DesR(z) = \{ s_i :(i,i+1) \in \inv(z)\}.
$
In turn let
$
\Cyc_\PP(z) = \Cyc_\ZZ(z) \cap (\PP\times \PP)
$.
The set
\be
\label{inv-fpf-eq}
\inv_\fpf(z) = \inv(z) - \Cyc_\ZZ(z)
\ee
 is finite with an even number of elements,
and is empty if and only if $z=\wfpf$.
We let
\be
\label{ell-fpf-eq}
\ellfpf(z) = \tfrac{1}{2}|\inv_\fpf(z)|
\qquand
\DesF(z) = \{ s_i  \in \DesR(z) : (i,i+1)\notin \Cyc_\ZZ(z)\}.
\ee
These definitions are related by the following proposition.

\begin{proposition} \label{fpfdes-prop}
If $z \in \F_\ZZ$
then
$
\ellfpf(szs)
=
\begin{cases}
\ellfpf(z)-1 &\text{if }s \in \DesF(z) \\
\ellfpf(z) &\text{if }s \in \DesR(z) -\DesF(z) \\
\ellfpf(z)+1&\text{if }s \in \{s_i : i \in \ZZ\} - \DesR(z).
\end{cases}
$
\end{proposition}

\begin{proof}
If $s \in \DesR(z) - \DesF(z)$, we have $szs = z$.
When $s_i \in \DesF(z)$, we see $ z(i) > z(i+1) \neq i$ so $\inv_\fpf(z) = \inv_\fpf(szs) \cup \{(i,i+1), (z(i+1),z(i))\}$.
Then $\ellfpf(z) = \ellfpf(szs) + 1$.
Finally, if $s \notin \DesR(z)$, we see $szs$ satisfies the previous case so $\ellfpf(z) = \ellfpf(szs) - 1$.
\end{proof}

Define $\cAfpf(z)$ for $z \in \F_\ZZ$ as the set of permutations $w \in S_\ZZ$ of minimal length
with $z=w^{-1} \wfpf w$.
This set is nonempty and finite, and its elements all have length $\ellfpf(z)$.
We define 
\be\label{crfpf-eq}
\cRfpf(z) = \bigsqcup_{w \in \cAfpf(z)} \cR(w)
\ee
to be the set of \emph{(reduced) fixed-point-free involution words} for $z$.

\begin{definition} \label{fS-def}
The \emph{FPF-involution Schubert polynomial} of $z \in \F_\infty$ is
$ \Sfpf_z \omdef= \sum_{w \in \cAfpf(z)} \fkS_w.$ 
\end{definition}

For $z \in \F_n$,  we set $\cAfpf(z) = \cAfpf(\ifpf(z))$ 
and $\Sfpf_z = \Sfpf_{\ifpf(z)}$.

\begin{example}
We have
$\ifpf(4321) = s_1s_2\wfpf s_2 s_1 = s_3s_2\wfpf s_2s_3$
and
$\cAfpf(4321) = \{312,1342\},$
so
\[\Sfpf_{4321} =  \fkS_{312} +  \fkS_{1342} =x_1^2 + x_1 x_2 + x_1 x_3 + x_2 x_3.\]
\end{example}

The polynomials $\Sfpf_z$
have the following characterization
via divided differences.

\begin{theorem}[{\cite[Corollary 3.13]{HMP1}}]
\label{fpfthm}
The FPF-involution Schubert polynomials $\{ \Sfpf_z\}_{z \in \F_\infty}$
are the unique family of homogeneous polynomials indexed by $\F_\infty$
such that if $i \in \PP$ and $s=s_i$ then
\be\label{fpf-eq}
\Sfpf_{\wfpf} = 1
\qquand
\partial_i \Sfpf_z = \begin{cases}
\Sfpf_{szs} &\text{if $s \in \DesR(z)$ and $(i,i+1)\notin\Cyc_\ZZ(z)$} \\
0&\text{otherwise}.\end{cases}
\ee
\end{theorem}

Wyser and Yong first considered these polynomials in \cite{WY},
where they were denoted $\Upsilon_{z; (\GL_{n}, \Sp_n)}$.
They showed, when $n$ is even, that the FPF-involution Schubert polynomials
indexed by $\F_n$ are cohomology representatives for the
$\Sp_n(\CC)$-orbit closures in 
the flag variety 
 $\Fl(n) = \GL_n(\CC)/B$, with $B\subset \GL_n (\CC)$ denoting the
Borel subgroup of lower triangular matrices.
%
The symmetric functions $\Ffpf_z$
are related to the polynomials $\Sfpf_z$ by the following identity.

 \begin{definition} \label{fF-def}
The  \emph{FPF-involution Stanley symmetric function} of $z \in \F_\ZZ$
is the power series
\[\Ffpf_z \omdef= \sum_{w \in \cAfpf(z)} F_w = \lim_{N\to \infty} \Sfpf_{z \gg 2N} \in \Lambda.\]
\end{definition}

\begin{lemma}\label{istab-lem}
If $z \in \F_\infty$ then $\Ffpf_z = \lim_{n\to \infty} \pi_{w_n} \Sfpf_z$.
\end{lemma}

\begin{proof}
This is immediate from Lemma~\ref{stab-lem}.
\end{proof}


\subsection{Schur $P$-functions}
\label{ss:schur-p}

Our main results will relate  $\Ffpf_z$ to the \emph{Schur $P$-functions} in $\Lambda$,
which were introduced  in work of Schur \cite{Schur} and have since
arisen in a variety of other contexts (see, e.g., \cite{BH, J16, P2}).
Good references for these symmetric functions
include 
\cite[\S6]{Stem} and \cite[{\S}III.8]{Macdonald2}.
For integers  $0\leq m \leq n$, let
\be\label{Gmn-eq}
G_{m,n} \omdef= \prod_{i \in [m]} \prod_{j \in [n-i]} \(1 +x_i^{-1} x_{i+j}\) \in \cL.
\ee
For a partition $\lambda = (\lambda_1,\lambda_2,\dots)$, let $\ell(\lambda)$
denote the largest index $i \in \PP$ with $\lambda_i\neq 0$.
The partition $\lambda$ is \emph{strict} if $\lambda_i \neq \lambda_{i+1}$
for all $i < \ell(\lambda)$.
Define $x^\lambda = x_1^{\lambda_1}x_2^{\lambda_2}\cdots x_{\ell}^{\lambda_\ell}$
where $\ell = \ell(\lambda)$.

\begin{definition}\label{schurp-def}
Let $\lambda$ be a strict partition
with $\ell=\ell(\lambda)$ parts.
The power series
\[P_\lambda \omdef=  \lim_{n\to \infty} \pi_{w_n} \(x^\lambda G_{\ell,n}\) \in \Lambda\]
is then a well-defined,
homogeneous symmetric function of degree $\sum_i \lambda_i$,
which one calls the \emph{Schur $P$-function} of $\lambda$.
\end{definition}

We present this slightly unusual definition of $P_\lambda$ for its
compatibility with Definition~\ref{F-def}.
The symmetric functions $P_\lambda$ may be described more concretely as generating functions for certain shifted tableaux \cite[Ex. (8.16$'$), {\S}III.8]{Macdonald2}.
The equivalence of the two definitions is explained in~\cite[Example 1, \S{III.8}]{Macdonald2}.

Whereas the Schur functions form a $\ZZ$-basis for $\Lambda$, the Schur $P$-functions
form a $\ZZ$-basis for the subring  $\Gamma = \QQ[ p_1,p_3,p_5,\dots ] \cap \Lambda$
generated by the odd-indexed power sum symmetric functions \cite[Corollary 6.2(b)]{Stem}.
Sagan \cite{sagan1987shifted} and Worley \cite{worley1984theory} showed independently that each Schur $P$-function $P_\lambda$ is itself Schur positive.
For more information about the positivity properties of the symmetric functions, see the discussion of \cite[Eq.\ (8.17), {\S}III.8]{Macdonald2} in Macdonald's book.

\section{Transition formulas}\label{invtrans-sect}

The \emph{Bruhat order} $<$ on $S_\ZZ$ is the weakest partial order with $w<wt$ when $w \in S_\ZZ$
and $t \in S_\ZZ$ is a transposition such that $\ell(w) < \ell(wt)$.
We define the \emph{Bruhat order} $<$ on $\F_\ZZ$ as the weakest partial order
with $z < tzt$ when $z \in \F_\ZZ$ and $t \in S_\ZZ$ is a transposition such that
$\ellfpf(z) < \ellfpf(tzt)$.
Rains and Vazirani's  results in \cite{RainsVazirani}
imply the following
theorem from \cite{HMP3}. 
\begin{theorem}[{\cite[Theorem 4.6]{HMP3}}]
\label{fpfbruhat-thm}
Let $n \in 2\PP$.
The following properties hold:
\ben
\item[(a)] $(\F_\ZZ,<)$ is a graded poset with rank function $\ellfpf$.
\item[(b)] If $y,z \in \F_n$ then $y\leq z$  holds in $(S_\ZZ,<)$
 if and only if $\ifpf(y) \leq \ifpf(z)$  holds in $(\F_\ZZ,<)$.

\item[(c)] Fix $y,z \in \F_\ZZ$ and $w \in \cAfpf(z)$.
Then $y\leq z$ if and only if some $v \in \cAfpf(y)$ has $v \leq w$.

 \een
\end{theorem}

Both $\ifpf(\F_n)$ and $\F_\infty$ are lower ideals in $(\F_\ZZ,<)$.
We write $y \lessdot_\F z$ for $y,z \in \F_\ZZ$ if  $ \{ w \in \F_\ZZ : y\leq w < z\}=\{y\}$.
If $y,z \in \F_n$ for some $n \in 2\PP$ and $\ifpf(y) \lessdot_\F \ifpf(z)$, then we write $y \lessdot_\F z$.
For example, 
 the set $\F_4$
  is totally ordered by $<$
 and we have
 \[\F_4 = \{ (1,2)(3,4) \lessdot_\F (1,3)(2,4) \lessdot_\F (1,4)(2,3) \}.\]

Let  $z \in \F_\ZZ$. Cycles
$(a,b),(i,j)\in \Cyc_\ZZ(z)$ with $a<i$
are \emph{crossing} if $a<i<b<j$ and \emph{nesting} if $a<i<j<b$.
One can check that $\ellfpf(z) = 2n+c$ where $n$ and $c$ are
the respective numbers of unordered
pairs of nesting and crossing cycles of $z$.
If $E\subset \ZZ$ has size $n \in \PP$ then we write 
$\phi_E$ and $\psi_E $
for the unique order-preserving bijections
$[n]\to E$ and $ E \to [n]$, and define 
\be
\label{st-eq}
[z]_E \omdef= \psi_{z(E)} \circ z \circ \phi_E \in S_n.
\ee
The operation $z \mapsto [z]_E$ is usually called \emph{standardization} or \emph{flattening}.


\begin{proposition}[{\cite[Corollary 2.3]{Bertiger}}]
\label{fpfbruhatcover-prop}
Let $y \in \F_\ZZ$. Fix integers $i<j$ and let $A = \{i,j,y(i),y(j)\}$ and $z=(i,j)y(i,j)$.
Then $\ellfpf(z) = \ellfpf(y)+1$ if and only if the following conditions hold:
\ben
\item[(a)]  One has $y(i)<y(j)$ but no $e \in \ZZ$ exists with  $i<e<j$ and $y(i) < y(e)< y(j)$.
\item[(b)] Either $[y]_A= (1,2)(3,4) \lessdot_\F [z]_A= (1,3)(2,4)$ or 
$[y]_A=(1,3)(2,4) \lessdot_\F [z]_A= (1,4)(2,3)$.

\een

\end{proposition}

\begin{remark}
If condition (a) holds then $(i,j) \notin \Cyc_\ZZ(y)$ so  necessarily  $|A|=4$. 
Condition (b) asserts that $[y]_A \lessdot_\F [z]_A$, which occurs if and only if 
$[y]_A $ and $ [z]_A$ coincide with
\[ \arcstart
{
*{.}  \arc{.6}{r}   & *{.}   & *{.} \arc{.6}{r}  & *{.}
}
\arcstop
\ \lessdot_\F \
\arcstart
{
*{.}  \arc{.8}{rr}   & *{.} \arc{.8}{rr}   & *{.} & *{.}
}
\arcstop
\qquord
\arcstart
{
*{.}  \arc{.8}{rr}   & *{.} \arc{.8}{rr}   & *{.} & *{.}
}
\arcstop
\ \lessdot_\F \
\arcstart
{
*{.}  \arc{1.0}{rrr}   & *{.} \arc{.4}{r}   & *{.} & *{.}
}
\arcstop.\]
In the first case $[(i,j)]_A \in \{(1,4),(2,3)\}$, and in the second
$[(i,j)]_A \in \{(1,2),(3,4)\}$.
\end{remark}

Given $y \in \F_\ZZ$ and $r \in \ZZ$,
let
\be\label{Pfpf-def}
\ba
\Pfpf^+(y,r) &\omdef=
\left\{
z\in \F_\ZZ : \ellfpf(z) = \ellfpf(y)+1 \text{ and } z = (r,j)y(r,j)\text{ for an integer }j>r
\right\},
\\
\Pfpf^-(y,r) &\omdef=
\left\{
z \in \F_\ZZ : \ellfpf(z) = \ellfpf(y)+1 \text{ and } z =(i,r)y(i,r)\text{ for an integer }i<r
\right\}.
\ea
\ee
These sets are both nonempty,
and 
if $z$ belongs to either of them then $y \lessdot_\F z$.
We can now state the transition formula for FPF-involution Schubert polynomials.

\begin{theorem}[{\cite[Theorem 4.17]{HMP3}}] \label{monkfpf-thm}
If $y \in \F_\infty$ and $(p,q) \in \Cyc_\PP(y)$
then
\[
(x_p+x_q) \Sfpf_y = \sum_{ z \in \Pfpf^+(y,q)} \Sfpf_{z} -  \sum_{z \in \Pfpf^-(y,p)} \Sfpf_{z}
\]
where we
 set $\Sfpf_z = 0$ for all $z \in \F_\ZZ-\F_\infty$.
 \end{theorem}

\begin{example}
Set $\Pfpf^\pm(y,r) = \Pfpf^\pm(\ifpf(y),r)$ if $y \in \F_n$.
For $y = (1,2)(3,7)(4,5)(6,8) \in \F_8$, 
\[
\ba \Pfpf^+(y,7) 
&= \{ (7,8)y(7,8)\} = \{(1,2)(3,8)(4,5)(6,7)\}
\\
\Pfpf^-(y,3) &= \{ (2,3)y(2,3) \} = \{ (1,3)(2,7)(4,5)(6,8) \}
\ea
\]
so
$(x_3+x_7)\Sfpf_{(1,2)(3,7)(4,5)(6,8)} = \Sfpf_{(1,2)(3,8)(4,5)(6,7)}-\Sfpf_{(1,3)(2,7)(4,5)(6,8)}$.
\end{example}

Taking limits and invoking Definition~\ref{fF-def} gives the following identity in $\Lambda$.

 \begin{theorem}\label{monkfpf-thm2}
 If $y \in \F_\ZZ$ and $(p,q) \in \Cyc_\ZZ(y)$ then
 $\sum_{z \in \Pfpf^-(y,p)} \Ffpf_{z}  = \sum_{ z \in \Pfpf^+(y,q)} \Ffpf_{z}.$
\end{theorem}

 \begin{proof}
We have $\Pfpf^\pm(y \gg 2N, r+2N) = \{ w \gg 2N : w \in \Pfpf^\pm(y,r)\}$
for  $y \in \F_\ZZ$ and $r,N \in \ZZ$,
so it follows that
$ \sum_{ z \in \Pfpf^+(y,q)} \Ffpf_{z} -  \sum_{z \in \Pfpf^-(y,p)} \Ffpf_{z} 
= \lim_{N\to \infty}  (x_{p+2N} + x_{q+2N}) \Sfpf_{y\gg 2N}  = 0$.
\end{proof}


\section{FPF-Grassmannian involutions}\label{fgrass-sect}

In this section we identify a class of ``Grassmannian'' elements of $\F_\ZZ$
for which $\Ffpf_z$ is a Schur $P$-function.
The \emph{(Rothe) diagram} of a permutation $w \in S_\infty$
is the set 
\be
\label{d-eq}
D(w) \omdef= \{ (i,j) \in \PP\times \PP : i < w^{-1}(j)\text{ and }j < w(i)\}.
\ee
Equivalently, $D(w) = \{ (i, w(j)) : (i,j) \in \inv(w)\}$
where 
\[\inv(w) \omdef= \{ (i,j) \in \ZZ \times \ZZ
: i<j \text{ and }w(i)>w(j)\}.
\] 
Following \cite[Section 3.2]{HMP1},
the \emph{(FPF-involution) diagram} of $z \in \F_\infty$ is the set
\be\label{dfpf-eq}
\Dfpf(z) \omdef= \{ (i,j) \in \PP\times \PP : j < i < z(j)\text{ and }j<z(i)\}.
\ee
One can check that $\Dfpf(z) = \{ (i,z(j))  : (i,j) \in \inv_\fpf(z),\ z(j)<i\}$. 

The \emph{code} of $w \in S_\infty$  is the 
sequence $c(w)=(c_1,c_2,c_3,\dots)$ where $c_i$ is the number of integers $j > i$ with
$w(i) > w(j)$.
The $i$th term of $c(w)$ is the number of
positions in the $i$th row of $D(w)$.
As in the introduction,
the \emph{(FPF-involution) code} of $z \in \F_\infty$ is the sequence
$\cfpf(z)= (c_1,c_2,\dots)$ in which $c_i$ is the number of positions
in the $i$th row of $\Dfpf(z)$, and the \emph{shape} of $z$ is the partition $\flambda(z)$ whose transpose is the partition that sorts $\cfpf(z)$.
For $z \in \F_n$ when $n \in 2\PP$, we define 
\[\Dfpf(z) \omdef= \Dfpf(\ifpf(z))
\qquand \cfpf(z) \omdef= \cfpf(\ifpf(z)).\]
Then $\Dfpf(z) $ is the subset of positions in $D(z)$ strictly below the diagonal.


The shifted shape of a strict partition $\mu$ is the set $\{ (i,i+j-1) \in \PP\times \PP : 1 \leq j \leq \mu_i \}$.
An involution $z $ in $\F_n$ or $\F_\infty$ is \emph{FPF-dominant} if
$\{ (i-1,j) : (i,j) \in \Dfpf(z)\}$ is
  the transpose of the shifted shape of a strict partition (which is necessarily $\flambda(z)$).
(We shift up since $\Dfpf(z)$
has no positions in row $i=1$.)
By contrast, a permutation is \emph{dominant} if it is merely 132-avoiding.

\begin{example}\label{fpf-dom-ex}
While $y = (1,8)(2,4)(3,5)(6,7) $ is FPF-dominant, $z = (1,3)(2,7)(4,8)(5,6)$ is not.
The corresponding diagrams are
   \begin{equation*}
    \Dfpf(y) = \begin{array}{cccccccc}
      \cdot & \cdot & \cdot & \cdot & \cdot & \cdot & \cdot & \times\\
      \circ  & \cdot & \cdot & \times & \cdot & \cdot & \cdot & \cdot\\
      \circ  & \circ & \cdot & \cdot & \times & \cdot & \cdot & \cdot\\
      \circ  & \times & \cdot & \cdot & \cdot & \cdot & \cdot & \cdot\\
      \circ  & \cdot & \times & \cdot & \cdot & \cdot & \cdot & \cdot\\
      \circ  & \cdot & \cdot & \cdot & \cdot & \cdot & \times & \cdot \\
      \circ  & \cdot & \cdot & \cdot & \cdot & \times & \cdot & \cdot\\
      \times  & \cdot & \cdot & \cdot & \cdot & \cdot & \cdot & \cdot
    \end{array} \quand 
    \Dfpf(z) = \begin{array}{cccccccc}
      \cdot & \cdot & \times & \cdot & \cdot & \cdot & \cdot & \cdot\\
      \circ  & \cdot & \cdot & \cdot & \cdot & \cdot & \times & \cdot\\
      \times  & \cdot & \cdot & \cdot & \cdot & \cdot & \cdot & \cdot\\
        \cdot  & \circ & \cdot & \cdot & \cdot & \cdot & \cdot & \times\\
      \cdot  & \circ & \cdot & \circ & \cdot & \times & \cdot & \cdot\\
      \cdot  & \circ & \cdot & \circ & \times & \cdot & \cdot & \cdot\\
      \cdot  & \times & \cdot & \cdot & \cdot & \cdot & \cdot & \cdot\\
      \cdot  & \cdot & \cdot & \times & \cdot & \cdot & \cdot & \cdot
    \end{array}
  \end{equation*}
where cells with $\circ$ are in $\Dfpf$,  $\times$ indicates a non-zero entry in the permutation matrix and $\cdot$ indicates a cell not in the diagram.
Observe that $\Dfpf$ consists of the positions
below the diagonal that are not weakly below any $\times$ and not weakly right of any $\times$.
The relevant codes are 
\[\cfpf(y) = (0,1,2,1,1,1,1,0)\qquand \cfpf(z) = (0,1,0,1,2,2,0,0),\]
and $\flambda(y) = (6,1)$ is the transpose of $(2,1,1,1,1,1)$.
The involution $y$ is not dominant (i.e. 132-avoiding) since in one-line notation $y=84523761$.
One can show that the only elements of $\F_{n}$ for $n \in \PP$ that are dominant in the classical sense are those of the
form $(1,n+1)(2,n+2)\cdots(n,2n)$. These involutions are all FPF-dominant.
\end{example}

The following generalizes \cite[Theorem~1.3]{HMP1}, which applies only when $z \in \F_n$ is dominant.
\begin{theorem}
\label{fpfdominant-thm}
If $z \in \F_\infty$ is FPF-dominant then
$
\Sfpf_z =  \prod_{(i,j) \in \Dfpf(z)} (x_i+x_j)
$.
\end{theorem} 

\begin{proof}
  For $z' \in \F_n$ we defined $\Sfpf_{z'} = \Sfpf_{\iota(z')}$, so we may as well assume $z \in \F_n$ for some $n$. Since $z = w_n$ is dominant, by~\cite[Theorem~1.3]{HMP1} we have
  \begin{equation*}
    \Sfpf_{w_n} = \prod_{\substack{1 \leq i < j \leq n \\ i+j \leq n}} (x_i + x_j).
  \end{equation*}
  Now assume $z \neq w_n$, and induct downward on $\ellfpf(z)$. Let $j \in [n]$ be minimal such that $z(j) < n-j+1$. The choice of $j$ implies $z(j)+1 \notin \{z(1), z(2), \ldots, z(j)\}$, so $z(z(j)+1) \notin [j]$. Setting $s = s_{z(j)}$, this shows $s \notin \DesR(z)$ and hence $\ellfpf(szs) = \ellfpf(z)+1$ by Proposition~\ref{fpfdes-prop}. Given that $z < zs < szs$, it is not hard to check that
  \begin{equation} \label{eq:diagram-change}
    D(szs) = D(z) \sqcup \{(z(j), j), (j, z(j))\}.
  \end{equation}
   If $z(j) < j$, then the minimality of $j$ implies $j = z(z(j)) = n-z(j)+1$, a contradiction; hence $z(j) > j$, so  \eqref{eq:diagram-change} implies $\Dfpf(szs) = \Dfpf(z) \sqcup \{(z(j), j)\}.$
   For example, if our involution is $z = (1,8)(2,7)(3,5)(4,6)$, then $j = 3$ and the diagrams of $z$ and $szs$ are drawn as follows:
   \begin{equation*}
    \Dfpf(z) = \begin{array}{cccccccc}
      \cdot & \cdot & \cdot & \cdot & \cdot & \cdot & \cdot & \times\\
      \circ  & \cdot & \cdot & \cdot & \cdot & \cdot & \times & \cdot\\
      \circ  & \circ & \cdot & \cdot & \times & \cdot & \cdot & \cdot\\
      \circ  & \circ & \circ & \cdot & \cdot & \times & \cdot & \cdot\\
      \circ  & \circ & \times & \cdot & \cdot & \cdot & \cdot & \cdot\\
      \circ  & \circ & \cdot & \times & \cdot & \cdot & \cdot & \cdot\\
      \circ  & \times & \cdot & \cdot & \cdot & \cdot & \cdot & \cdot\\
      \times  & \cdot & \cdot & \cdot & \cdot & \cdot & \cdot & \cdot
    \end{array} \qquad 
    \Dfpf(szs) = \begin{array}{cccccccc}
      \cdot & \cdot & \cdot & \cdot & \cdot & \cdot & \cdot & \times\\
      \circ  & \cdot & \cdot & \cdot & \cdot & \cdot & \times & \cdot\\
      \circ  & \circ & \cdot & \cdot & \cdot & \times & \cdot & \cdot\\
      \circ  & \circ & \circ & \cdot & \times & \cdot & \cdot & \cdot\\
      \circ  & \circ & \circ & \times & \cdot & \cdot & \cdot & \cdot\\
      \circ  & \circ & \times & \cdot & \cdot & \cdot & \cdot & \cdot\\
      \circ  & \times & \cdot & \cdot & \cdot & \cdot & \cdot & \cdot\\
      \times  & \cdot & \cdot & \cdot & \cdot & \cdot & \cdot & \cdot
    \end{array}
  \end{equation*}
  On the left, $\times$ is a point of the form $(i, z(i))$ and $\circ$ indicates an element of $\Dfpf(z)$, i.e., a point above and left of a $\times$ and below the main diagonal.
The picture on the right follows the same conventions with $z$ replaced by $szs$.
  
   Let $\lambda = \flambda(z)$ be the shape of $z$. 
   Since $z(j) > j$ and $z(i) = n-i+1$ for $i < j$, drawing a picture makes clear that $\lambda_j = z(j)-j-1$ and $\lambda_i = n-2i$ for $i < j$. 
   The previous paragraph therefore shows that $szs$ is FPF-dominant with shape $\flambda(szs) = (\lambda_1, \cdots, \lambda_{j-1}, \lambda_j + 1, \lambda_{j+1}, \ldots)$. By induction,
   \begin{equation*}
   \Sfpf_{szs} = \prod_{(a,b) \in \Dfpf(szs)} (x_a + x_b) = (x_{z(j)} + x_j) \prod_{(a,b) \in \Dfpf(z)} (x_a + x_b).
   \end{equation*}
   We claim that $\prod_{(a,b) \in \Dfpf(z)} (x_a + x_b)$ is symmetric in the variables $x_{z(j)}$ and $x_{z(j)+1}$. First, $z(j) > j$ forces column $z(j)$ of $\Dfpf(z)$ to be empty, so any variable $x_{z(j)}$ or $x_{z(j)+1}$ 
   in the product
   comes from a factor $x_a + x_b$ with $(a,b) = (z(j), b) \in \Dfpf(z)$. 
   The inner corners of $\lambda$ (the cells rightmost in their row and bottommost in their column) appear in columns $n-1,n-2,\ldots,n-j+1,z(j)-1,\ldots$ from right to left. 
   Thus, since $z(j)-1 < z(j)< z(j)+1 \leq n-j+1$, columns $z(j)$ and $z(j)+1$ of $\lambda$ have the same length---in the figure above, these two columns appear (transposed) as rows $5$ and $6$ of $\Dfpf(z)$. This implies that $(z(j), b) \in \Dfpf(z)$ if and only if $(z(j)+1, b) \in \Dfpf(z)$, which proves the claim. Now
   \begin{align*}
    \Sfpf_z = \partial_{z(j)} \Sfpf_{szs} &= \partial_{z(j)} \left[(x_{z(j)} + x_j) \prod_{(a,b) \in \Dfpf(z)} (x_a + x_b)\right]    \\
    &= \partial_{z(j)} (x_{z(j)} + x_j) \prod_{(a,b) \in \Dfpf(z)} (x_a + x_b) \\
     &= \prod_{(a,b) \in \Dfpf(z)} (x_a + x_b). 
   \end{align*}
\end{proof}

The \emph{lexicographic order} on $S_\infty$ is the total order
induced by identifying $w \in S_\infty$ with its one-line representation $w(1)w(2)w(3)\cdots$.
For $z$ in $\F_n$ or $\F_\infty$, we let $\beta_{\min}(z)$ denote the  lexicographically minimal element of $\cAfpf(z)$.
The next lemma follows from  \cite[Theorem 6.22]{HMP2}.

\begin{lemma}\label{minatomfpf-lem}
Suppose $z \in \F_\infty$ and $\Cyc_\PP(z) = \{ (a_i,b_i) : i \in \PP\}$ where $a_1<a_2<\cdots$.
The lexicographically  minimal element $\beta_{\min}(z)\in \cAfpf(z)$  is the inverse of
the permutation  whose one-line representation is   $a_1b_1a_2b_2a_3b_3\cdots$.
\end{lemma}

The same statement with  ``$a_1b_1a_2b_2\cdots$'' replaced by
``$a_1b_1a_2b_2\cdots a_n b_n$'' holds if $z \in \F_{2n}$.

\begin{example}
If $z =(1,4)(2,3) \in \F_4$ then $a_1b_1a_2b_2 = 1423$ and
$\beta_{\min}(z)=1423^{-1} = 1342$.
\end{example}

Typically $\Dfpf(z) \neq D(\beta_{\min}(z))$, but the analogous statement holds for codes.

\begin{lemma}[{\cite[Lemma 3.8]{HMP1}}]
\label{codef-lem}
If $z \in \F_\infty$ then 
 $\cfpf(z) = c(\beta_{\min}(z))$.
 \end{lemma}

A pair $(i,j) \in \ZZ\times \ZZ$ is an \emph{FPF-visible inversion} of $z \in \F_\ZZ$
if $i<j$ and $z(j) < \min\{i,z(i)\}$.
These are precisely the involutions corresponding to the cells of $\Dfpf(z)$, as noted in Definition~\ref{d:diag}.

\begin{lemma}\label{desf-lem0}
The set of FPF-visible inversions of $z \in \F_\infty$ is $\inv(\beta_{\min}(z))$.
\end{lemma}

\begin{proof}
Suppose $(i,j) \in \ZZ \times \ZZ$ is an FPF-visible inversion of $z \in F_\infty$.
Either $z(j)<i <z(i)$ or $z(j) < z(i) < i$, and in both cases $j$ appears before $i$
in the one-line representation of $\beta_{\min}(z)^{-1}$ so $(i,j) \in \inv(\beta_{\min}(z))$.
Since $|\inv(\beta_{\min(z)})| = \ellfpf(z) = |\Dfpf(z)|$, this completes our proof. \end{proof}

If $(i,i+1)$
is an FPF-visible inversion of $z \in \F_\ZZ$,
then $i \in \ZZ$ is an \emph{FPF-visible descent}.
Let 
\be
\label{desif-eq}
\DesIF(z) \omdef = \{ s_i : i \in \ZZ \text{ is an FPF-visible descent of $z$}\} \subset \DesF(z).
\ee
Since  $s_i \in \DesR(w)$ for $w \in S_\ZZ$ if and only if $(i,i+1) \in \inv(w)$,
the following is immediate.

\begin{lemma}\label{desf-lem}
If $z \in \F_\infty$ then  $\DesIF(z) = \DesR(\beta_{\min}(z))$.
\end{lemma}

The \emph{essential set}  of a subset $D\subset \PP\times \PP$ is  the set $\Ess(D)$ of
positions $(i,j) \in D$ such that $(i+1,j) \notin D $
and $(i,j+1) \notin D$.
The following is similar to \cite[Lemma 4.14]{HMP4}.

\begin{lemma}\label{desf-lem2} 
For $z \in \F_\infty$, the $i$th row of $\Ess(\Dfpf(z))$ is nonempty
if and only if $s_i \in \DesIF(z)$.
\end{lemma}

\begin{proof}
If $s_i \in \DesIF(z)$ then $(i,z(i+1)) \in \Dfpf(z)$
but all positions of the form $(i+1,j) \in \Dfpf(z)$ have $j < z(i+1)$,
so the $i$th row of $\Ess(\Dfpf(z))$ is nonempty.
Conversely, if the $i$th row of this set is nonempty, 
then there is some $(i,j) \in \Dfpf(z)$ with $(i+1,j) \notin \Dfpf(z)$.
This holds only if $j=z(k)$ for some $k>i$ with $z(i) > z(k)$
and $i > z(k) \geq z(i+1)$, in which case $s_i \in \DesIF(z)$.
\end{proof}

A permutation $w \in S_\infty$ is \emph{$n$-Grassmannian} if $\DesR(w) = \{s_n\}$.

\begin{proposition}
\label{p:fpf-grass}
For $z \in \F_\infty$ and $n \in \PP$,
the following are equivalent:
\ben
\item[(a)] $\DesIF(z) = \{ s_n\}$.

\item[(b)] $\cfpf(z)$ has the form $(0,c_2,\dots,c_n,0,0,\dots)$ where
$ c_2 \leq \cdots \leq c_n\neq0$.

\item[(c)] $ \Ess(\Dfpf(z))$ is nonempty and contained in $\{ (n,j) : j\in \PP\}$.

\item[(d)] The lexicographically minimal atom $\beta_{\min}(z) \in \cAfpf(z)$ is $n$-Grassmannian.
\een
\end{proposition}

\begin{proof}
We have (a) $\Leftrightarrow$ (d) by Lemma~\ref{desf-lem} and (a) $\Leftrightarrow$ (c) by Lemma~\ref{desf-lem2}.
Finally, Lemma~\ref{codef-lem} implies that (b) $\Leftrightarrow$ (d)
since $w \in S_\infty$ is $n$-Grassmannian
if and only if the first $n$ terms of $c(w)$ are weakly increasing and 
the remaining entries are 0.
\end{proof}

The preceding conditions
suggest a natural concept of a ``Grassmannian'' fixed-point-free
involution, but
this definition turns out to be slightly too restrictive.
Define  $\I_\ZZ \omdef= \{ w \in S_\ZZ : w=w^{-1}\}$.
Consider the  maps
$
\Fmap : \I_\ZZ \to \F_\ZZ $
and
$ \Imap : \F_\ZZ \to \I_\ZZ
$ given as follows.

\begin{definition}\label{Fmap-def}
For $y \in \I_\ZZ$, let $m$ be any even integer
with $m < i $ for all $i \in \supp(y)$,
write $\phi$ for the order-preserving bijection $\ZZ \to \ZZ\setminus \supp(y)$ with $\phi(0) = m$,
and define $\Fmap(y)$ as the unique element of $\F_\ZZ$ with
$
\Fmap(y)(i) = y(i)$ for $i\in \supp(y)$
and
$\Fmap(y)\circ \phi = \phi \circ \wfpf .$
\end{definition}

We use the symbol $\Fmap$ to denote this map since $\Fmap(y)$ is formed by ``arc-ifying'' the 
matching that represents $y$, i.e., by adding in edges to pair up all isolated vertices.

We have
$\Fmap(y) = \ifpf(y)$ for $y \in \F_n$.
The involution $\Fmap(z)$ is formed from $z$ by turning every pair of adjacent fixed points into a cycle;
there are two ways of doing this, and we choose the way that  makes $(2i-1,2i)$ into a cycle for
all sufficiently
large $i \in \ZZ$. For example, we have
\[
\Fmap\(
\arcstart
{
*{.} \arc{1}{rrrr}&
*{.}   &
*{.} &
*{.}   & 
*{.}   &
*{.}  &
*{.}  \arc{1}{rr}&
*{.}&
*{.}
\\
*{1} &
*{2} & 
*{3} &
*{4} &
*{5} &
*{6} &
*{7} &
*{8} &
*{9} 
}
\arcstop
\)
\ =\ 
\arcstart
{
*{\dots}&
*{.}  \arc{.5}{r}&
*{.}&
*{.} \arc{1}{rrrr}&
*{.} \arc{.5}{r}  &
*{.} &
*{.} \arc{1}{rr}  & 
*{.}   &
*{.}  &
*{.}  \arc{1}{rr}&
*{.}  \arc{1}{rr}&
*{.}&
*{.} &
*{.} \arc{.5}{r}&
*{.} &
*{\dots}
\\
*{}&
*{} &
*{}&
*{1} &
*{2} & 
*{3} &
*{4} &
*{5} &
*{6} &
*{7} &
*{8} &
*{9} &
*{} &
*{} &
*{}
}
\arcstop
\]

\begin{definition}\label{I-def}
For $z \in \F_\ZZ$, 
define $\Imap(z) \in \I_\ZZ$ as the involution whose nontrivial cycles 
are precisely the pairs $(p,q) \in \Cyc_\ZZ(z)$ for which there exists $(a,b ) \in \Cyc_\ZZ(z)$
with $p<b<q$.
\end{definition}

We use the symbol $\Imap$ to denote this map since $\Imap(z)$ is formed by removing all
``trivial'' arcs from the
matching that represents $z$.

The permutation $\Imap(z)$ is the involution that restricts to the same map as $z$ on its support,
and  whose fixed points are the integers $i \in \ZZ$ such that
$\max\{ i,z(i)\} < z(j)$ for all $j \in \ZZ$ with $\min\{i,z(i)\} < j < \max\{i,z(i)\}$.
 For example, we have
\[
\Imap\(
\arcstart
{
*{\dots}&
*{.}  \arc{.5}{r}&
*{.}&
*{.} \arc{1}{rr}&
*{.} \arc{1}{rrrr}  &
*{.} &
*{.} \arc{.5}{r}  & 
*{.}   &
*{.}  &
*{.}  \arc{1}{rr}&
*{.} \arc{1}{rr}&
*{.} &
*{.}  &
*{.}  \arc{.5}{r}&
*{.}&
*{\dots}
\\
*{}&
*{} &
*{}&
*{1} &
*{2} & 
*{3} &
*{4} &
*{5} &
*{6} &
*{7} &
*{8} &
*{9} &
*{10} &
*{} &
*{}
}
\arcstop
\)
\ =\ 
\arcstart
{
*{.} &
*{.} \arc{1}{rrrr}  &
*{.} &
*{.}   & 
*{.}   &
*{.}  &
*{.}&
*{.}  \arc{1}{rr}&
*{.}&
*{.} 
\\
*{1} &
*{2} & 
*{3} &
*{4} &
*{5} &
*{6} &
*{7} &
*{8} &
*{9} &
*{10}
}
\arcstop
\]
We see in these examples that $\Imap$ and $\Fmap$ restrict to maps $\F_\infty \to \I_\infty$ and
$\I_\infty \to \F_\infty$.

\begin{proposition}
Let $z \in \F_\ZZ$. Then $\Imap(z) = 1$ if and only if $z =\wfpf$.
\end{proposition}

\begin{proof}
If $z \neq \wfpf$ and $i$ is the largest integer such that $i < z(i) \neq i+1$, then
necessarily $z(i+1) < z(i)$, so  $(i, z(i))$ is a nontrivial cycle of $\Imap(z)$,
which is therefore not the identity.
\end{proof}

\begin{proposition}\label{semiinverse-prop}
The composition $\Fmap\circ \Imap$ is the identity map $\F_\ZZ \to \F_\ZZ$.
\end{proposition}

\begin{proof}
Fix $z \in \I_\infty$. 
Let $\cC$ be the set of cycles $(p,q) \in \Cyc_\ZZ(z)$
such that $p$ and $q$ are fixed points in $\Imap(z)$.
By definition, if $(p,q)$ and $(p',q')$ are distinct elements of $\cC$
then  $p<q<p'<q'$ or $p'<q'<p<q$.
The claim that $\Fmap\circ \Imap(z) = z$
is a straightforward consequence of this fact.
\end{proof}

An involution $y \in \I_\ZZ$ is \emph{I-Grassmannian} if $y=1$ or
$y = (\phi_1, n+1)(\phi_2,n+2)\cdots (\phi_r,n+r)$
for some integers $r \in \PP$ and  $ \phi_1 < \phi_2<\dots < \phi_r \leq n$.
See \cite[Proposition-Definition 4.16]{HMP4} for several equivalent characterizations of such involutions.

\begin{definition} 
\label{d:fpf_grass}
An involution $z \in \F_\ZZ$ is \emph{FPF-Grassmannian} if 
$\Imap(z)\in \I_\ZZ$ is I-Grassmannian.
\end{definition}

Define an element of $\F_n$
to be FPF-Grassmannian if its image under $\ifpf : \F_n \to \F_\infty\subset \F_\ZZ$ is FPF-Grassmannian.

\begin{remark}
The sequence $(g^\fpf_{n})_{n\geq 1} = (1,3,12,41,124,350,952,2540,\dots)$
with $g^\fpf_n$ the number of FPF-Grassmannian elements of $\ifpf(\F_n)\subset \F_\ZZ$
seems unrelated to any sequence in \cite{OEIS}.
\end{remark}

Suppose $z \in \F_\ZZ-\{\wfpf\}$ is FPF-Grassmannian,
so that 
\[\Imap(z)=(\phi_1, n+1)(\phi_2,n+2)\cdots (\phi_r,n+r) \in \I_\infty\]
for integers $r \in \PP$ and  $\phi_1 < \phi_2<\dots < \phi_r \leq n $.
Recall from the introduction that $\flambda(z)$ is the transpose of the partition given by sorting $\cfpf(z)$.

\begin{lemma}
\label{l:grass_shape}
In the notation just given, it holds that $\flambda(z) = (n-\phi_1,n-\phi_2,\dots,n-\phi_r)$.
\end{lemma}

\begin{proof}
The definitions of $\Dfpf(y)$, $\cfpf(y)$ and $\flambda(y)$ make sense even when $y \in \I_\ZZ$.
Let $y = \Imap(z)$. It is easy to check that the only nonempty columns of $\Dfpf(y)$ are $\phi_1,\phi_2,\dots,\phi_r$ and that the $\phi_i$th column is $\{(\phi_i+1,\phi_i),(\phi_i+2,\phi_i),\dots,(n,\phi_i)\}$.
Therefore  $\flambda(y) = (n-\phi_1,n-\phi_2,\dots,n-\phi_r)$, since 
sorting $\cfpf(y)$ gives the transpose of this partition.

Fix positive integers $i<k$ and suppose
$(i,k)$ is a cycle in $z$ that is not a cycle in $y$, so that $y(i) = i$ and $y(k) = k$.
Suppose $i<j<k$. From the definition of $\Imap$, it follows that
$(j,i) \in \Dfpf(z) \setminus \Dfpf(y)$
and $j = \phi_l$ for some $l\in [r]$.
Therefore, we have $(k,j) \in \Dfpf(y) \setminus \Dfpf(z)$, so 
\[
\Dfpf(z) \cap [i,k]^2= \{(p,j) \in \PP \times \PP: i \leq j < p < k\}\]
and 
\[ \Dfpf(y) \cap [i,k]^2= \{(p,j) \in \PP \times \PP: i < j < p \leq k\}.
\]
If $p$ is an integer with $i \leq p \leq k$ then
\[
\{q < i: (p,q) \in \Dfpf(z)\} = \{q < i: (p,q) \in \Dfpf(y)\} = \{l: \phi_l < i\}.
\]
With $\cfpf(z) = (c_1(z),c_2(z),\dots)$ and $\cfpf(y) = (c_1(y),c_2(y),\dots)$, we deduce that $c_j(z) = c_{j+1}(y)$ for $i \leq j < k$ and $c_k(z) = c_i(y)$.
When $j$ is not between the endpoints of some cycle $(i,k)$ in $z$ but not $y$, we have $c_j(y) = c_j(z)$.
Therefore $\cfpf(z)$ and $\cfpf(y)$ are the same multisets, so $\flambda(z) = \flambda(y)$.
\end{proof}

\begin{example} 
Consider $z = (1,4)(2,6)(3,7)(5,8)$ and $y = \Imap(z) = (2,6)(3,7)(5,8)$.
Then
\[
\Dfpf(z) = \begin{array}{cccccccc}
    \cdot &  \cdot &  \cdot &  \times &  \cdot &  \cdot &  \cdot &  \cdot\\
    \circ &  \cdot &  \cdot &  \cdot &  \cdot &  \times &  \cdot &  \cdot\\
    \circ &  \circ &  \cdot &  \cdot &  \cdot &  \cdot &  \times &  \cdot\\
    \times &  \cdot &  \cdot &  \cdot &  \cdot &  \cdot &  \cdot &  \cdot\\
    \cdot &  \circ &  \circ &  \cdot &  \cdot &  \cdot &  \cdot &  \times\\
    \cdot &  \times &  \cdot &  \cdot &  \cdot &  \cdot &  \cdot &  \cdot\\
    \cdot &  \cdot &  \times &  \cdot &  \cdot &  \cdot &  \cdot &  \cdot\\
    \cdot &  \cdot &  \cdot &  \cdot &  \times &  \cdot &  \cdot &  \cdot
 \end{array}
 \]
 and
 \[
 \Dfpf(y) = \begin{array}{cccccccc}
    \times &  \cdot &  \cdot &  \cdot &  \cdot &  \cdot &  \cdot &  \cdot\\
    \cdot &  \cdot &  \cdot &  \cdot &  \cdot &  \times &  \cdot &  \cdot\\
    \cdot &  \circ &  \cdot &  \cdot &  \cdot &  \cdot &  \times &  \cdot\\
    \cdot &  \circ &  \circ &  \times &  \cdot &  \cdot &  \cdot &  \cdot\\
    \cdot &  \circ &  \circ &  \cdot &  \cdot &  \cdot &  \cdot &  \times\\
    \cdot &  \times &  \cdot &  \cdot &  \cdot &  \cdot &  \cdot &  \cdot\\
    \cdot &  \cdot &  \times &  \cdot &  \cdot &  \cdot &  \cdot &  \cdot\\
    \cdot &  \cdot &  \cdot &  \cdot &  \times &  \cdot &  \cdot &  \cdot
 \end{array}.
\]
The positions marked $\times$ in the respective diagrams are those of the form $(i,y(i))$ or $(i,z(i))$. 
We have $\cfpf(z) = (0,1,2,0,2,0,0)$ while $\cfpf(y) = (0,0,1,2,2,0,0)$.
In addition, we observe that $c_1(z) = c_2(y)$, $c_2(z) = c_3(y)$, and $c_3(z) = c_4(y)$, as predicted in the argument for Lemma~\ref{l:grass_shape}.
\end{example}

Given integers $a,b \in \PP$ with $a< b$, define
$\partial_{b,a} = \partial_{b-1} \partial_{b-2}\cdots \partial_a$
and
$\pi_{b,a} = \pi_{b-1} \pi_{b-2}\cdots \pi_a$.
 For $a,b \in \PP$ with $a\geq b$, set
$\partial_{b,a} = \pi_{b,a} = \id$.

\begin{lemma}\label{fgrass-lem4}
Maintain the preceding setup, but assume $z $ is an FPF-Grassmannian element of $\F_\infty-\{\wfpf\}$ so that 
 $1\leq \phi_1 <\phi_2<\dots < \phi_r \leq n $. Then
$ \Sfpf_z = \pi_{\phi_1,1} \pi_{\phi_2,2} \cdots \pi_{\phi_r,r} \( x^{\flambda(z)} G_{r,n}\).$
\end{lemma}

\begin{proof}
The proof 
depends on the following claim, which is proved as \cite[Lemma 2.2]{HMP4}:
\begin{claim}
If $a\leq b$ and  $f \in \cL$ are such that $\partial_i f = 0$ for $a< i < b$,
then
$\pi_{b,a} f = \partial_{b,a}\(x_a^{b-a}f\).$
\end{claim}
If $c_1<c_2 < \dots < c_k $ are the fixed points in $[n]$ of $\Imap(z)$, then $k$ is even and
we have
$(c_1,c_2), (c_3,c_4), \dots, (c_{k-1},c_k) \in \Cyc_\ZZ(z)$.
Hence if $\phi_i=i$ for all $i \in [r]$
then $z$ is FPF-dominant and
\[\Dfpf(z) = \{ (i+j,i) : i \in [r]\text{ and } j \in [n-i]\}.\]
In this case the lemma reduces to the formula
$\Sfpf_z = x_1^{n-1} x_2^{n-2}\cdots x_r^{n-r} G_{r,n}$
which follows from Theorem \ref{fpfdominant-thm}.

Alternatively, suppose there exists $i \in [r]$ such that $i<\phi_i$.
Assume $i$ is minimal with this property.
Then $\Sfpf_z = \partial_{\phi_i,i} \Sfpf_v$ for the FPF-Grassmannian
involution $v \in \F_\infty$ with 
\[
\Imap(v)=(1, n+1)(2,n+2)\cdots(i,n+i)(\phi_{i+1},n+i+1)(\phi_{i+2},n+i+2)\cdots (\phi_r,n+r).
\]
By induction 
$ \Sfpf_v = \pi_{\phi_{i+1},i+1} \pi_{\phi_{i+2},i+2} \cdots \pi_{\phi_r,r} \( x^{\flambda(v)} G_{r,n}\).$
Since $x^{\flambda(v)} = x_i^{\phi_i-i}x^{\flambda(z)}$
and since multiplication by $x_i$
commutes with $\pi_{j}$ when $i<j$,
it follows from the claim that
\[
\ba
\Sfpf_z &= \partial_{\phi_i,i} \Sfpf_v\\& = 
\partial_{\phi_i,i} ( 
x_i^{\phi_i-i} \pi_{\phi_{i+1},i+1} \pi_{\phi_{i+2},i+2} \cdots \pi_{\phi_r,r} ( x^{\flambda(z)} G_{r,n})
)
\\&=
\pi_{\phi_i,i} \pi_{\phi_{i+1},i+1} \pi_{\phi_{i+2},i+2} \cdots \pi_{\phi_r,r} ( x^{\flambda(z)} G_{r,n})
.\ea\]
The last expression is $\pi_{\phi_1,1}  \cdots \pi_{\phi_r,r} ( x^{\flambda(z)} G_{r,n})$ since we assume $\pi_{\phi_1,1} = \dots = \pi_{\phi_{i-1},i-1} = \id.$
\end{proof}

\begin{theorem}\label{fgrass-thm}
If $z \in \F_\ZZ$ is FPF-Grassmannian, then $\Ffpf_z = P_{\flambda(z)}$.
\end{theorem}

\begin{proof}
Since $\Ffpf_z = \Ffpf_{z \gg N}$ for $N \in 2\ZZ$,
we may assume that $z \in \F_\infty$ and that 
 $\Imap(z) $ is I-Grassmannian. 
 Since $\pi_{w_n} \pi_i = \pi_{w_n}$ for $i \in [n-1]$,
Lemma~\ref{fgrass-lem4} implies
that if $\flambda(z)$ has $r$ parts and $n\geq r$
then
$ \pi_{w_n} \Sfpf_z = \pi_{w_n} \(x^{\flambda(z)} G_{r,n}\)$.
Now take the limit  as $n\to \infty$ and apply Lemma~\ref{istab-lem}.
\end{proof}

Let us clarify the difference between FPF-Grassmannian
involutions and elements of $\F_\ZZ$ with at most one FPF-visible descent.
Define $\I_\infty \omdef= S_\infty \cap \I_\ZZ$ and for each $y \in \I_\infty$
let 
\be
\label{desi-eq}
\DesI(y) \omdef= \{ i \in \ZZ :
z(i+1) \leq \min\{i,z(i)\}.
\ee
Elements of $\DesI(y)$ are \emph{visible descents} of $y$.

\begin{lemma}\label{prodef3-lem}
Let $z \in \F_\infty$ and $E = \{ i \in \PP : |z(i)-i| \neq 1\}$.
Suppose $y \in \I_\infty$ is the involution with $y(i) = z(i)$ if $i \in E$ and $y(i)=i$
otherwise. Then
  $z = \Fmap(y)$ and
$\DesIF(z) = \DesI(y)$.

\end{lemma}

\begin{proof}
It is evident that $z = \Fmap(y)$. Suppose $s_i \in \DesI(y)$. 
Since $y(i+1) \neq i$ for all $i \in \PP$ by definition, we must have  $y(i+1) < \min\{i,y(i)\}$,
so $i+1 \in E$, and therefore either  $i \in E$ or $z(i) = i-1$. It follows in either case that
$z(i+1) < \min\{i,z(i)\}$ so $s_i \in \DesIF(z)$.
Conversely, suppose $s_i \in \DesIF(z)$ so that $i+1 \in E$. If $ i \in E$ then
$s_i \in \DesI(y)$ holds immediately,
and if $i \notin E$ then
$z(i+1)<z(i)=i-1$, in which case $y(i+1) = z(i+1)<i=y(i)$ so $s_i \in \DesI(y)$.
\end{proof}

In our previous work, we showed that 
$y \in \I_\ZZ$ is I-Grassmannian if and only if 
$|\DesI(y)| \leq 1$
\cite[Proposition-Definition 4.16]{HMP4}.
Using this fact, we deduce the following:

\begin{proposition}
An involution $z \in \F_\ZZ$ has $|\DesIF(z)| \leq 1$ 
if and only if 
$z$ is FPF-Grassmannian and $\flambda(z)$
is a strict partition whose consecutive
parts each differ by odd numbers.
\end{proposition}

\begin{proof}
We may assume that
 $z \in \F_\infty-\{\wfpf\}$. 
If 
 $z$ is FPF-Grassmannian and the consecutive parts of $\flambda(z)$ 
 differ by odd numbers then one can check that $|\DesIF(z)| \leq 1$.
Conversely,
define $y \in \I_\infty$ as in Lemma \ref{prodef3-lem}
so that $z= \Fmap(y)$.
We have $\DesIF(z) = \DesI(y) = \{s_n\}$ if and only if 
$y=(\phi_1, n+1)(\phi_2,n+2)\cdots (\phi_r,n+r)$ for integers $r \in \PP$
and $0 = \phi_0 < \phi_1 < \phi_2<\dots < \phi_r \leq n$.
If $y$ has this form then each
$\phi_i - \phi_{i-1}$ is  necessarily odd,
and  $\Imap(z) = y$ or $\Imap(z) =(\phi_2,n+2)(\phi_3,n+3)\cdots (\phi_r,n+r)$, 
so
$z$ is FPF-Grassmannian and the consecutive parts of $\flambda(z)$
 differ by odd numbers.
\end{proof}

\begin{remark}
Using the previous result, one can show that the number 
 $k_n$ of elements of $\F_n$ with at most one FPF-visible descent
 satisfies the recurrence
$k_{2n} = 2k_{2n-2}+2n-3$ for $n \geq 2$. 
The corresponding sequence $(k_{2n})_{n\geq 1} = (1, 3, 9, 23, 53, 115, 241, 495,\dots)$ 
is  \cite[A183155]{OEIS}.
\end{remark}

\section{Schur $P$-positivity}\label{fpf-pos-sect}

In this section we describe a recurrence for expanding 
$\Ffpf_z$ into FPF-Grassmannian summands,
and use this to deduce that each $\Ffpf_z$ is Schur $P$-positive.
Our strategy is similar to the one used in \cite[\S4.2]{HMP4}, though with some
added technical complications.

Order the set $\ZZ\times \ZZ$ lexicographically.
Recall that  $(i,j) \in \ZZ\times \ZZ$ is an {FPF-visible inversion} of $z \in \F_\ZZ$
if $i<j$ and $z(j) < \min\{ i, z(i)\}$,
and that
 $i \in \ZZ$ is an {FPF-visible descent} of $z$ if $(i,i+1)$
is an FPF-visible inversion.
By Lemma \ref{desf-lem}, every $z \in \F_\ZZ-\{\wfpf\}$
has an FPF-visible descent.

\begin{lemma}
Let $z \in \F_\ZZ-\{\wfpf\}$ and suppose $j \in \ZZ$ is the smallest integer such that $ z(j)<j-1$.
Then $j-1$ is the minimal FPF-visible descent of $z$. 
\end{lemma}

\begin{proof}
By hypothesis, either $z(j) < j-2 = z(j-1)$ or $z(j) < j-1 < z(j-1)$, so $j-1$
is an FPF-visible descent of $z$.
If $k-1$ is another FPF-visible descent of $z$, then $z(k) < k-1$ so $j \leq k$.
\end{proof}

\begin{lemma}\label{dichotomy-lem}
Suppose $z \in \F_\ZZ-\{\wfpf\}$. Let $(q,r) \in \ZZ\times \ZZ$ be the lexicographically maximal FPF-visible inversion of $z$. Suppose $m$ is the largest even integer such that $z(m) \neq m-1$.
Then:
\begin{enumerate}
\item[(a)] The number $q$ is the maximal FPF-visible descent of $z$.
\item[(b)] The number $r$
is the maximal integer with $z(r) < \min\{q,z(q)\}$.
\item[(c)] It holds that $
z(q+1) < z(q+2) < \dots < z(m) \leq q
$.
\item[(d)] Either $z(q) < q < r \leq m$ or $q < z(q) = r + 1 = m$.
\end{enumerate}
\end{lemma}

\begin{proof}
Since $(q+1,r)$ is not an FPF-visible inversion of $z$, we must have $\min\{q+1,z(q+1)\} \leq z(r) < \min\{q,z(q)\}$.
These inequalities can only hold if $z(q+1) < q+1$,
so $q$ is an FPF-visible descent of $z$.
Since $(i,i+1)$ is not an FPF-visible inversion of $z$ for any $i>q$, we conclude that $q$ is the maximal FPF-visible descent of $z$.
This prove part (a).
Parts (b) and (c) follow similarly from the assumption that $(q,r)$ is the lexicographically maximal FPF-visible inversion.

If $z(q) < q$, then $z(q) < r \leq m$ since $(q,r)$ is an FPF-visible inversion.
Assume $q<z(q)$. To prove (d), it remains to 
show that $z(q) = r+1=m$.
It cannot hold that $r<z(q)-1$, since then either $(q,r+1)$ or $(r+1,z(q))$ 
would be an FPF-visible inversion of $z$, contradicting the maximality of $(q,r)$. 
It also cannot hold that $z(q) < r$, as then 
$(z(q),r)$ would be an FPF-visible inversion of $z$. Hence $r=z(q)-1$.
If $j > z(q)$, then 
since $z(i) < q$ for all $q <i < z(q)$ and since $(z(q), j)$ cannot be an FPF-visible inversion of $z$,
we must have $z(j) > z(q)$.
From  this observation and the fact that 
 $z$ has no FPF-visible descents greater than $q$, 
 we deduce that $z(j) = \wfpf(j)$ for all $j > z(q)$, which implies that $z(q) = m$ as required.
\end{proof}

\begin{definition}\label{fpsi-def}
Let $\fpsi :\F_\ZZ-\{\wfpf\} \to \F_\ZZ$ be the map  $\fpsi : z\mapsto (q,r)z(q,r)$
where $(q,r)$
 is the maximal FPF-visible inversion of $z$.
\end{definition}

\begin{remark}\label{fpf-remark}
Suppose $z \in \F_\ZZ-\{\wfpf\}$ has maximal FPF-visible inversion $(q,r)$. Let $p=z(r)$ and
$y=\fpsi(z) = (q,r)z(q,r)$
and  write $m$ for the largest even integer such that $z(m) \neq m-1$.
The two cases of Lemma~\ref{dichotomy-lem}~(d) correspond to the following pictures:
\ben
\item[(a)]
If $z(q) < q < r\leq m$ then $y$ and $z$ may be represented as
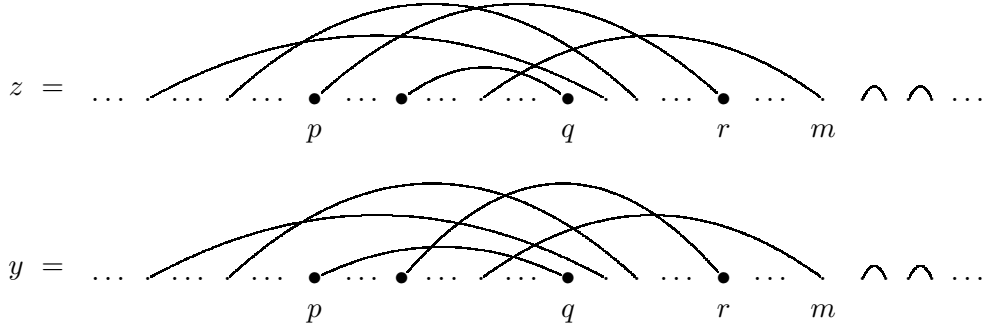
\[
\ba
\\
\\
z &\ =\
\arcstartc{.2}
{
*{\dots}&
*{.} \arc{2}{rrrrrrrrrrr}&
*{\dots}&
*{.} \arc{3}{rrrrrrrrrr}&
*{\dots}&
*{\bullet} \arc{3}{rrrrrrrrrr} &
*{\dots} &
*{\bullet} \arc{1}{rrrr} &
*{\dots} &
*{.} \arc{2}{rrrrrrrr}&
*{\dots} &
*{\bullet} &
*{.} &
*{.} &
*{\dots} &
*{\bullet} &
*{\dots} &
*{.} &
*{.} \arc{.4}{r} &
*{.} &
*{.} \arc{.4}{r} &
*{.} &
*{\dots}
\\
&
&
&
&
&
p
&
&
&
&
&
&
q
&
&
&
&
r
&
&
m
}
\arcstop
\\
\\
\\
y&\ =\  
\arcstartc{.2}
{
*{\dots}&
*{.} \arc{2}{rrrrrrrrrrr}&
*{\dots}&
*{.} \arc{3}{rrrrrrrrrr}&
*{\dots}&
*{\bullet} \arc{1}{rrrrrr} &
*{\dots} &
*{\bullet} \arc{3}{rrrrrrrr} &
*{\dots} &
*{.} \arc{2}{rrrrrrrr}&
*{\dots} &
*{\bullet} &
*{.} &
*{.} &
*{\dots} &
*{\bullet} &
*{\dots} &
*{.} &
*{.} \arc{.4}{r} &
*{.} &
*{.} \arc{.4}{r} &
*{.} &
*{\dots}
\\
&
&
&
&
&
p
&
&
&
&
&
&
q
&
&
&
&
r
&
&
m
}
\arcstop
\ea
\]
We have
 $z(q+1) < z(q+2) < \dots <z(r) < z(q)$, and if $r<m$ then $z(q)<z(r+1) < z(r+2) < \dots < z(m)  < q$.

\item[(b)] 
 If $q<z(q) =r+1=m$ then $y$ and $z$ may be represented as
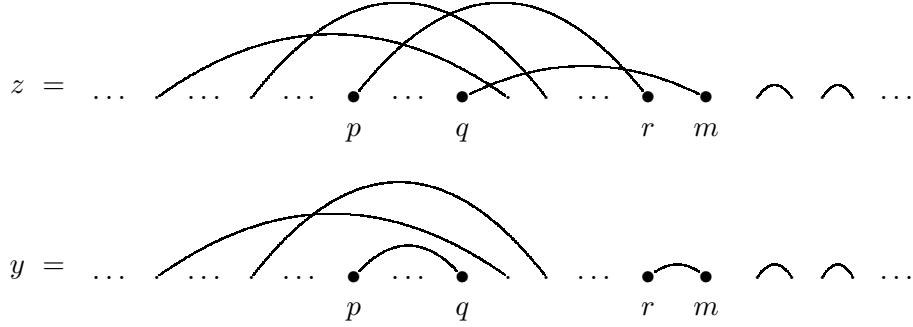
\[
\ba
\\
\\
z &\ =\  
\arcstartc{.3}
{
*{\dots}&
*{.} \arc{2}{rrrrrrr}&
*{\dots}&
*{.} \arc{3}{rrrrrr}&
*{\dots}&
*{\bullet} \arc{3}{rrrrrr} &
*{\dots} &
*{\bullet} \arc{1}{rrrrr} &
*{.} &
*{.} &
*{\dots} &
*{\bullet} &
*{\bullet} &
*{.} \arc{.4}{r} &
*{.} &
*{.} \arc{.4}{r} &
*{.} &
*{\dots}
\\
&
&
&
&
&
p
&
&
q
&
&
&
&
r
&
m
&
}
\arcstop
\\
\\
\\
y
&\ =\ 
\arcstartc{.3}
{
*{\dots}&
*{.} \arc{2}{rrrrrrr}&
*{\dots}&
*{.} \arc{3}{rrrrrr}&
*{\dots}&
*{\bullet} \arc{1}{rr} &
*{\dots} &
*{\bullet}] &
*{.} &
*{.} &
*{\dots} &
*{\bullet}  \arc{.4}{r}  &
*{\bullet} &
*{.} \arc{.4}{r} &
*{.} &
*{.} \arc{.4}{r} &
*{.} &
*{\dots}
\\
&
&
&
&
&
p
&
&
q
&
&
&
&
r
&
m
&
}
\arcstop
\ea
\]
Here, we have $z(q+1) < z(q+2) < \dots < z(r)  =p < q$, so $z(i)<q$ whenever $p<i<q$.
\een
\end{remark}

Recall the definition of $\beta_{\min}(z)$ from Lemma~\ref{minatomfpf-lem}.

\begin{proposition}
If $(q,r)$ is the maximal FPF-visible inversion of $z \in \F_\infty-\{\wfpf\}$ and
$w = \beta_{\min}(z)$ is the minimal element of $\cAfpf(z)$,
then $w(q,r) = \beta_{\min}(\fpsi(z))$ is the minimal atom of $\fpsi(z)$.
\end{proposition}

\begin{proof}
Let $\Cyc_\PP(z) = \{ (a_i,b_i) : i \in \PP\}$ and $\Cyc_\PP(\fpsi(z)) = \{ (c_i,d_i): i \in \PP\}$
where $a_1<a_2<\dots$ and $c_1<c_2<\dots$.
By Lemma \ref{minatomfpf-lem}, it suffices to show that interchanging $q$ and $r$  in the word
$a_1b_1a_2b_2\cdots$ gives  $c_1d_1c_2d_2\cdots$,
which is   straightforward  from  Remark~\ref{fpf-remark}.
\end{proof}

Recall the definition of the sets $\Pfpf^+(y,r)$ and $\Pfpf^-(y,r)$ from \eqref{Pfpf-def}.

\begin{lemma}\label{ft-lem00}
If  $z \in \F_\ZZ-\{\wfpf\}$ has  maximal FPF-visible inversion $(q,r)$ then
 $\Pfpf^+(\fpsi(z),q)= \{z\}$.
\end{lemma}

\begin{proof}
This holds by Proposition~\ref{fpfbruhatcover-prop},
Remark~\ref{fpf-remark}, and the definitions of $\fpsi(z)$ and $\Pfpf^+(y,q)$.
\end{proof}

   
For $z \in \F_\ZZ$  let
\be\label{fT-eq}   \fT_1(z) \omdef= \begin{cases}
\varnothing &\text{if  $z$ is FPF-Grassmannian} \\
 \Pfpf^-(y, p) &\text{otherwise}
 \end{cases}
 \ee
 where in the second case, we define $y=\fpsi(z)$ and $p =y(q)$
 where  $q$ is the maximal FPF-visible descent of $z$.

\begin{definition}\label{fT-def}
The \emph{FPF-involution Lascoux-Sch\"utzenberger tree} $\fT(z)$ of $z \in \F_\ZZ$
is the
tree with root $z$, in which the children of any vertex $v \in \F_\ZZ$
are the elements of $\fT_1(v)$.
\end{definition}

\begin{remark}
As the name suggests, our definition is inspired by the classical
construction of the \emph{Lascoux-Sch\"utzenberger tree}
for ordinary Stanley symmetric functions; see \cite{LS,Little} or \cite[\S4.2]{HMP4}.
\end{remark}

For $z \in \F_n$ we define $\fT(z) = \fT(\ifpf(z))$.
A given involution is allowed to correspond to more than one vertex in $\fT(z)$.
All vertices $v$ in $\fT(z)$ satisfy $\ellfpf(v) = \ellfpf(z)$ by construction,
so if $z \neq \wfpf$ then $\wfpf$ is not a vertex in $\fT(z)$.
An example tree $\fT(z)$ is shown in Figure~\ref{ftree-fig}.

\begin{figure}[h]
\[
\begin{tikzpicture}
%
  \node (b) at (0,0) {
\boxed{\arcstartc{.2}
{
{ }
\\
\\
*{.} \arc{.4}{r}  &
*{.} &
  *{.}  \arc{1.2}{rrrr} &
*{.} \arc{.6}{rr} &
*{\circ} \arc{1.2}{rrrrr} &
*{.} &
*{.} &
*{.} \arc{.8}{rrr}&
*{\bullet} \arc{.8}{rrr}&
*{.} &
*{\bullet} &
*{.} 
}
\arcstop}
};
  \node (c) at (5,0) {
};
\node (d) at (-5,-2) {
};
\node (e) at (0,-2) {
\boxed{\arcstartc{.2}
{
{ }
\\
\\
*{.}  \arc{.4}{r} &
*{.} &
  *{.} \arc{1.2}{rrrrr} &
*{.} \arc{.6}{rr} &
*{\circ}  \arc{1.2}{rrrrr}&
*{.} &
*{.} \arc{.6}{rr}&
*{.} &
*{\bullet} &
*{\bullet} &
*{.} \arc{.4}{r}&
*{.} 
}
\arcstop}
};
\node (f) at (5,-2) {
\boxed{\arcstartc{.2}
{
{ }
\\
\\
*{.}  \arc{.4}{r} &
*{.} &
  *{.} \arc{1.2}{rrrr} &
*{.}   \arc{1.2}{rrrr}&
*{.} \arc{1.2}{rrrrr}&
*{\circ} \arc{.6}{rrr}&
*{.} &
*{.} &
*{\bullet} &
*{\bullet} &
*{.} \arc{.4}{r}&
*{.} 
}
\arcstop}
};
\node (g) at (-5,-4) {
\boxed{\arcstartc{.2}
{
{ }
\\
\\
*{.} \arc{.4}{r}  &
*{.} &
  *{.} \arc{1.2}{rrrrrr} &
*{.}  \arc{.6}{rr}&
*{\circ} \arc{.6}{rrr}&
*{.} &
*{\bullet} \arc{.6}{rrr}&
*{.} &
*{\bullet} &
*{.} &
*{.} \arc{.4}{r}&
*{.} 
}
\arcstop}
};
\node (h) at (0,-4) {
\boxed{\arcstartc{.2}
{
{ }
\\
\\
*{.} \arc{.4}{r}  &
*{.} &
  *{.}\arc{1.2}{rrrrr}  &
*{.}  \arc{1.2}{rrrrr}&
*{.} \arc{.4}{r} &
*{.} &
*{\bullet} \arc{.6}{rrr}&
*{.} &
*{\bullet} &
*{.} &
*{.} \arc{.4}{r}&
*{.} 
}
\arcstop}
};
\node (i) at (5,-4) {
\boxed{\arcstartc{.2}
{
{ }
\\
\\
*{.}   \arc{.4}{r}&
*{.} &
  *{.} \arc{.8}{rrrr} &
*{.}  \arc{1.2}{rrrrr}&
*{.} \arc{.8}{rrr}&
*{\circ}  \arc{.8}{rrrr}&
*{.} &
*{\bullet} &
*{\bullet} &
*{.} &
*{.} \arc{.4}{r}&
*{.} 
}
\arcstop}
};
\node (j) at (-5,-6) {
\boxed{\arcstartc{.2}
{
{ }
\\
\\
*{.}  \arc{.6}{rr} &
*{\circ} \arc{1.2}{rrrrrr}&
  *{.} &
*{.} \arc{.6}{rr} &
*{.} \arc{1.2}{rrrr}&
*{.} &
*{\bullet} \arc{.8}{rrr}&
*{.} &
*{\bullet} &
*{.} &
*{.} \arc{.4}{r} &
*{.} 
}
\arcstop}
};
\node (k) at (0,-6) {
};
\node (l) at (5,-6) {
\boxed{\arcstartc{.2}
{
{ }
\\
\\
*{.}   \arc{.4}{r}&
*{.} &
  *{.} \arc{1.2}{rrrrr}  &
*{.}  \arc{.8}{rrr}&
*{.} \arc{1.2}{rrrr}&
*{\circ}  \arc{1.2}{rrrr}&
*{\bullet} &
*{\bullet} &
*{.} &
*{.} &
*{.} \arc{.4}{r}&
*{.} 
}
\arcstop}
};
  \draw  [-]
  (b) -- (e)
  (b) -- (f)
  (e) -- (g)
  (e) -- (h)
  (f) -- (i)
  (g) -- (j)
  (i) -- (l)
;
\end{tikzpicture}
\]
\caption{The tree $\fT(z)$ for $z =(1,2)(3,7)(4,6)(5,10)(8,11)(9,12)\in \F_{12}\hookrightarrow \F_\ZZ$.
We draw all vertices as elements of $\F_{12}\subset \I_{12}$
for convenience. 
The maximal FPF-visible inversion of each vertex is marked with $\bullet$,
and the minimal FPF-visible descent is marked with $\circ$ (when this is not also maximal).
 By Theorem \ref{fgrass-thm} and Corollary \ref{ft-cor1},
we have
$\Ffpf_{z} = P_{(5,2)} + P_{(4,3)} + P_{(4,2,1)}$.
}
\label{ftree-fig}
\end{figure}

\begin{corollary}\label{ft-cor1}
Suppose $z \in \F_\ZZ$ is a fixed-point-free involution that is not FPF-Grassmannian,
whose maximal FPF-visible descent is $q \in \ZZ$. The following identities then hold:
\ben
\item[(a)] $\Sfpf_z =(x_p+x_q) \Sfpf_{y} +  \sum_{v \in \fT_1(z)} \Sfpf_v$ where $y = \fpsi(z)$ and $p = y(q)$.

\item[(b)] $\Ffpf_z = \sum_{v \in \fT_1(z)} \Ffpf_v$.
\een
\end{corollary}

\begin{proof}
The result follows from Theorems \ref{monkfpf-thm} and \ref{monkfpf-thm2} and Lemma \ref{ft-lem00}.
\end{proof}

We would like to show that
the intervals between the minimal and maximal FPF-visible 
descents of the vertices in $\fT(z)$ form a descending chain as one moves down the tree. 
This fails, however:
a child in the tree may have  strictly smaller FPF-visible descents than its parent.
A similar property does hold if 
we instead consider the visible descents of the image of $z \in \F_\ZZ$
under the map $\Imap : \F_\ZZ \to \I_\ZZ$ from Definition~\ref{I-def}.
Recall that a visible descent for $y \in \I_\ZZ$
is an integer $i \in \ZZ$ with
$z(i+1) \leq \min\{i,z(i)\}$.
The following is  \cite[Lemma 4.24]{HMP4}.

\begin{lemma}[See \cite{HMP4}]
\label{minvisdes-lem}
Let $z \in \I_\ZZ-\{1\}$ and suppose $j \in \ZZ$ is the smallest integer such that $z(j)<j$.
Then $j-1$ is the minimal visible descent of $z$.
\end{lemma}

\begin{lemma}\label{fpf-min-des-lem}
Let $z \in \F_\ZZ-\{\wfpf\}$ and suppose $(i,j) \in \Cyc_\ZZ(z)$
is the cycle with $j$ minimal such that $i < b < j$ for some $(a,b) \in \Cyc_\ZZ(z)$.
Then $j-1$ is the minimal visible descent of $\Imap(z)$.
\end{lemma}

\begin{proof}
The claim follows by the preceding lemma since $j$ is minimal such that $\Imap(z)(j)<j$.
\end{proof}

\begin{lemma}\label{fpf-vis-lem1}
Let $z \in \F_\ZZ$. A number $ i \in \ZZ$ is a visible descent of $\Imap(z)$ if and only if
one of the following conditions holds:
\ben
\item[(a)] $z(i+1) < z(i) < i$.
\item[(b)] $z(i) < z(i+1) < i$ and $\{ t \in \ZZ : z(i) < t < i\}  \subset \{ z(t) : i < t \}$.
\item[(c)] $z(i+1) < i < z(i)$ and $\{ t \in \ZZ: z(i+1) < t < i+1\} \not \subset \{ z(t) : i+1 <t \}$.
\een
\end{lemma}

\begin{proof}
It is straightforward to check that
$i \in \ZZ$ is a visible descent of $\Imap(z)$ if and only if either
(a) $z(i+1) < z(i) < i$;
(b) $z(i) < z(i+1) < i$ and $i$ is a fixed point of $\Imap(z)$;
or (c) $z(i+1) < i < z(i)$ and $i+1$ is not a fixed point of $\Imap(z)$.
The given conditions are equivalent to these statements.
\end{proof}

\begin{corollary}\label{fpf-vis-cor1}
Let $y,z \in \F_\ZZ$ and $i,j \in \ZZ$ with $i<j$. Suppose $y(t) = z(t)$ for all integers $t>i$.
Then $j$ is a visible descent of $\Imap(y)$ if and only if $j$ is a visible descent of $\Imap(z)$.
\end{corollary}

\begin{proof}
By Lemma~\ref{fpf-vis-lem1},
 whether or not $j$ is a visible descent of $\Imap(z)$ depends only
on the action of $z$ on integers greater than or equal to $j$. 
\end{proof}

\begin{corollary}\label{fpf-vis-cor2}
Let $z \in \F_\ZZ$ and
suppose $i$ is a visible descent of $\Imap(z)$. Then either $i$ or $i-1$ is an FPF-visible descent of $z$.
Therefore, if $j$ is the maximal FPF-visible descent of $z$, then $i \leq j+1$.
 \end{corollary}

\begin{proof}
It follows from Lemma~\ref{fpf-vis-lem1} that $i$ is an FPF-visible descent of $z$ unless
$z(i) < z(i+1) < i$ and $\{ t \in \ZZ : z(i) < t < i\}  \subset \{ z(t) : i < t \}$, in which case
 $i-1$ is an FPF-visible descent of $z$.
\end{proof}

The following statement is the first of two key technical lemmas in this section.

\begin{lemma}\label{crux-min-fpf-des-lem}
Let $y \in \F_\ZZ-\{\wfpf\}$ and $(p,q) \in \Cyc_\ZZ(y)$,
and suppose $z = (n,p)y(n,p) \in \Pfpf^-(y, p)$.
\ben
\item[(a)] If $i \in \ZZ\setminus \{ n, y(n), p, q\}$ is such that $\Imap(y)(i) = i$, then
$\Imap(z)(i) = i$.

\item[(b)] If $j$ and $k$ are the minimal visible descents of $\Imap(y)$ and  $\Imap(z)$
and $j \leq q-1$, then $j \leq k$.
\een
\end{lemma}

\begin{remark}
Part (b) is false if $j \geq q$: consider $y = (6,7)\wfpf (6,7)$ and $(n,p,q) = (2,3,4)$.
There is no analogous inequality governing the minimal FPF-visible descents of $y$ and $z$.
\end{remark}

\begin{proof}
Since $y \lessdot_\F z = (n,p)y(n,p)\in \Pfpf^-(y, p) $,
it follows from Proposition~\ref{fpfbruhatcover-prop}
that either $y(n) < n < p < q$, in which case $ n < p<z(p) < q=z(n)$ and $y$ and $z$ 
correspond to the diagrams
\be\label{case1-eq}
y \ =\
\arcstartc{.1}
{
*{\dots}&
*{\bullet} \arc{1}{rr}&
*{\dots}&
*{\bullet} &
*{\dots}&
*{\bullet} \arc{1}{rr}&
*{\dots}&
*{\bullet}&
*{\dots}
\\
&
&
&
n
&
&
p
&
&
q
}
\arcstop
\qquand
z \ =\  
\arcstartc{.1}
{
*{\dots}&
*{\bullet} \arc{1}{rrrr}&
*{\dots}&
*{\bullet}  \arc{1}{rrrr}&
*{\dots}&
*{\bullet}&
*{\dots}&
*{\bullet}&
*{\dots}
\\
&
&
&
n
&
&
p
&
&
q
}
\arcstop
\ee
 or 
$n <p < y(n) < q$, in which case $ n < p<z(p) < q=z(n)$ and
 we instead have 
\be
\label{case2-eq}
y  \ =\  
\arcstartc{.1}
{
*{\dots}&
*{\bullet} \arc{1}{rrrr}&
*{\dots}&
*{\bullet}  \arc{1}{rrrr}&
*{\dots}&
*{\bullet}&
*{\dots}&
*{\bullet}&
*{\dots}
\\
&
n
&
&
p
&
&
&
&
q
}
\arcstop
\qquand
z  \ =\  
\arcstartc{.1}
{
*{\dots}&
*{\bullet} \arc{1}{rrrrrr}&
*{\dots}&
*{\bullet}  \arc{.5}{rr}&
*{\dots}&
*{\bullet}&
*{\dots}&
*{\bullet}&
*{\dots}
\\
&
n
&
&
p
&
&
&
&
q
}
\arcstop
\ee
Let $A = \{n, y(n),p,q\} = \{ n,p,z(p), q\}$ and note that $y(i) =z(i)$ for all $ i \in \ZZ\setminus A$.
Suppose $(a,b) \in \Cyc_\ZZ(y)$ is  such that $b \notin A$
and $b < y(i)$ for all $a<i<b$, so that $a$ and $b$ are both fixed points of $\Imap(y)$.
Then $(a,b)$ is also a cycle of $z$,
and to prove part (a)
it suffices to check that $b<z(i)$ for all $i \in A$ with $a<i<b$.
This holds if $i \in \{n, y(n)\}$ since then
$y(i) < z(i)$, and we cannot have $a<q<b$ since $y(q) < q$.
Suppose $a<p<b$; it remains to show that $b<z(p)$.
Since $b < y(i)$ for all $a<i<b$ by hypothesis,
it follows that if $y$ and $z$ are as in \eqref{case1-eq} then $n<a<p<b<q$,
and that if $y$ are $z$ are as in \eqref{case2-eq} then $a<p<b<y(n)$.
The first of these cases cannot occur in view of Proposition~\ref{fpfbruhatcover-prop}(a),
since $y \lessdot_\F z$.
In the second case $y(n) = z(p)$ so $b<z(p)$ as needed.

To prove part (b), note that $\wfpf \notin \{y, z\}$
so neither  $\Imap(y)$ nor $\Imap(z)$ is the identity.
Let $j$ and $k$ be the minimal visible descents of $\Imap(y)$ and $\Imap(z)$
and assume $j \leq q-1$.
Write $S_y$ for the set of integers $i \in  \ZZ\setminus A$ such that $\Imap(y)(i) < i $,
and let $T_y = S_y \setminus A$ and $U_y = S_y \cap A$.
Define $S_z$, $T_z$, and $U_z$ similarly.
Lemma~\ref{minvisdes-lem} implies that $j \leq k$ if and only if $\min S_y \leq \min S_z$.
Since $j\leq q-1$ we have $\min S_y \leq q$.
It follows from part (a) that $T_z \subset T_y$, so $\min T_y \leq \min T_z$.

There are two cases to consider. First suppose
 $y(n) < n < p< q$ and $z(p) < n < p < q = z(n)$. It is then evident from \eqref{case1-eq}
that $\{q\} \subset U_z \subset \{p,q\}$. Since $\min S_y \leq q$ by hypothesis,
  to prove 
that $\min S_y \leq \min S_z$
it suffices to show that if $p \in U_z$ then $\min S_y < p$.
 Since $y \lessdot_\F z$,
 neither $y$ nor $z$ can have any cycles $(a,b)$ with $y(n) < a  < p$ and $n<b<p$.
 It follows that if $p \in U_z$ then $y$ and $z$ share a cycle $(a,b)$ with either
 (i) $a<b$ and $ y(n) < b < n$,
 or (ii) $a<y(n) < n < b < p$.
 If (i) occurs then $n \in U_y$ while if
 (ii) occurs then $\min T_y < p$, so $\min S_y < p$ as desired.

Suppose instead that $n < p < y(n) < q$ and $ n < p<z(p) < q=z(n)$.
In view of \eqref{case2-eq}, we then have $\{q\} \subset U_z \subset \{ y(n), q\}$.
As $ \min S_y \leq q$, to prove that $\min S_y \leq \min S_z$
it now suffices to show that if $y(n) \in U_z$ then $y(n) \in U_y$.
This implication is clear from \eqref{case2-eq}, since  if $y(n)=z(p) \in U_z$
then $y$ and $z$ must share a cycle $(a,b)$ with $a<b$ and $p<b<y(n)$.
\end{proof}

\begin{lemma}\label{fpf-min-des-lem2}
Let $y \in \F_\ZZ-\{\wfpf\}$ and $(p,q) \in \Cyc_\ZZ(y)$
and suppose $z = (q,r)y(q,r) \in \Pfpf^+(y, q)$.
The involution $\Imap(y)$ has a visible descent less than $q-1$ if and only if $\Imap(z)$ does,
and in this case the minimal visible descents of $\Imap(y)$ and $\Imap(z)$ are equal.
\end{lemma}

\begin{proof}
Let $\cC_w$ for $w \in \F_\ZZ$ be the set of cycles $(a,b) \in \Cyc_\ZZ(w)$ with $b < q$.
By Lemma~\ref{fpf-min-des-lem}, the set $\cC_w$ determines whether or not
$\Imap(w)$ has a visible descent less than $q-1$ and, when this occurs, the value of
$\Imap(w)$'s smallest visible descent.
Since $q < r$ we have $\cC_y = \cC_z$, so the result follows.
\end{proof}

Our second key technical lemma is the following.

\begin{lemma}\label{crux-fpf-des-lem}
Suppose $z \in \F_\ZZ$ is not FPF-Grassmannian, so that $\fpsi(z)\neq \wfpf$.
Let $(q,r)$ be the maximal FPF-visible inversion of $z$ and define $y = \fpsi(z) = (q,r)z(q,r)$.
\ben
\item[(a)] The maximal visible descent of $\Imap(z)$ is $q$ or $q+1$.
\item[(b)] The maximal visible descent of $\Imap(y)$ is at most $q$.
\item[(c)] The minimal visible descent of $\Imap(y)$ is equal to that of $\Imap(z)$, and is at most $q-1$.
\een
\end{lemma}

\begin{proof}
Adopt the notation of Remark~\ref{fpf-remark}. To prove the first two parts,
let $j$ and $k$ be the maximal visible descents of $\Imap(y)$ and $\Imap(z)$, respectively.
In case (a) of Remark~\ref{fpf-remark},  it follows by inspection that $j\leq q = k$,
with equality 
unless $r=q+1$ and there exists at least one cycle $(a,b) \in \Cyc_\ZZ(z)$ such that  $p<b<q$.
In case (b) of Remark~\ref{fpf-remark}, one of the following occurs:
\begin{itemize}
\item If $p=q-1=r-2$, then $j<q-1<k=q+1$.
\item If $p=q-1<r-2$, then $j=q$ and $k \in \{q,q+1\}$.
\item If $p<q-1$, then $j=k=q$.
\end{itemize}
We conclude  that $j\leq q$ and $k \in \{q,q+1\}$ as required.

Let $j$ and $k$ now be the minimal visible descents of $\Imap(y)$ and $\Imap(z)$, respectively.
Part (c) is immediate from Lemmas~\ref{ft-lem00} and \ref{fpf-min-des-lem2} if $j<q-1$ or $k<q-1$,
so assume that $j$ and $k$ are both at least $q-1$.
Suppose $z(q)<q <r \leq m$ so that we are in case (a) of Remark~\ref{fpf-remark},
when $q$ is the maximal visible descent of $\Imap(z)$. Since $z$ is not FPF-Grassmannian,
we must have $k=q-1$, so by Lemma~\ref{fpf-min-des-lem} there exists
$(a,b) \in \Cyc_\ZZ(z)$ with $z(q) < b < q$.
Since  $y(q) = p < z(q)$, it follows that $j \leq q-1$;
as the reverse inequality holds by hypothesis, we get
 $j=k=q-1$ as desired.

Suppose instead that we are in case (b) of Remark~\ref{fpf-remark}.
Since $q<z(q)$, it cannot hold that $q-1$ is a visible descent of $\Imap(z)$,
so we must have $k \geq q$.
As $z$ is not FPF-Grassmannian, it follows from part (a) that
$k=q$ and that $q+1$ is the maximal visible descent of $\Imap(z)$.
This is impossible, however, since we can only have $k=q$ if there exists $(a,b) \in \Cyc_\ZZ(z)$
with $z(q+1) < b < q+1$,
while $q+1$ can only be a visible descent of $\Imap(z)$ if no such cycle exists.
\end{proof}

\begin{lemma}\label{fpf-last-lem1}
Suppose $z \in \F_\ZZ$ is not FPF-Grassmannian and $v \in \fT_1(z)$.
Let $i$ and $j$ be the minimal and maximal visible descents of $\Imap(z)$.
If $d$ is a visible descent of $\Imap(v)$, then $i\leq d \leq j$.
\end{lemma}

\begin{proof}
Let $(q,r)$ be the maximal FPF-visible descent of $z$, set $y = (q,r)z(q,r)= \fpsi(z)$
and $p=y(q)=z(r)$, and let $n<p<q$ be the unique integer such that $v=(n,p)y(n,p)$.
Since $y \lessdot_\F v$, it must hold that $y(n) < q$, so $v(t) = y(t)$ for all $t>q$.
The maximal visible descent of $\Imap(y)$ is at most $q \leq j$ by Lemma~\ref{crux-fpf-des-lem},
 so the same is true of the maximal visible descent of $\Imap(v)$ by Corollary~\ref{fpf-vis-cor1}.
 On the other hand, the minimal visible descent of $\Imap(y)$ is $i\leq q-1$ by Lemma~\ref{crux-fpf-des-lem},
so  by Lemma~\ref{crux-min-fpf-des-lem} the minimal visible descent of $\Imap(v)$ is at least $i$.
\end{proof}

For any $z \in \F_\ZZ$, let $\fT_0(z) \omdef= \{ z\}$
and define $\fT_n(z) \omdef= \bigcup_{v \in \fT_{n-1}(z)} \fT_1(v)$ for $n\geq 1$.

\begin{lemma}\label{fpf-last-lem2}
Suppose $z \in \F_\ZZ$ and $v \in \fT_1(z)$.
Let $(q,r)$ be the maximal FPF-visible inversion of $z$,
and let $(q_1,r_1)$ be any FPF-visible inversion of  $v$. Then $q_1<q$ or $r_1<r$.
Hence, if $n\geq r-q$ then the maximal FPF-visible descent of every element of $\fT_n(z)$ is
strictly less than $q$.
\end{lemma}

\begin{proof}
It is considerably easier to track the FPF-visible inversions of $z$ and $v$ than
the visible inversions of $\Imap(z)$ and $\Imap(v)$, and this result follows essentially by inspecting
 Remark~\ref{fpf-remark}. In more detail, let $y = \fpsi(z) = (q,r)z(q,r)$ and $p=z(r) = y(q)$.
Since $y \lessdot_\F v = (n,p)y(n,p)$ for some $n<p$, we must have 
$v(i) = y(i)$ for all $i>q$, 
and so it is apparent 
 from Remark~\ref{fpf-remark}
that $q_1 \leq q$.
If $q_1=q$, then necessarily $v(q) < p<v(i)$ for all $i\geq r$, and it follows that $r_1<r$.
 \end{proof}

 \begin{theorem}\label{f-finite-thm}
 The FPF-involution Lascoux-Sch\"utzenberger tree $\fT(z)$ is finite for $z \in \F_\ZZ$,
 and  $\Ffpf_z = \sum_v \Ffpf_v$ where the sum
  is over the finite set of leaf vertices $v$ in $\fT(z)$.
 \end{theorem}

\begin{proof}
By induction,
 Corollary~\ref{fpf-vis-cor2}, and Lemmas \ref{fpf-last-lem1} and \ref{fpf-last-lem2},
we deduce that for a sufficiently large $n$
either $\fT_n(z) = \varnothing$ or all elements of $\fT_n(z)$ are FPF-Grassmannian,
whence $\fT_{n+1}(z) = \varnothing$.
The tree $\fT(z)$ is therefore finite,  so the identity $\Ffpf_z = \sum_v \Ffpf_v$ holds by
Corollary \ref{ft-cor1}.
\end{proof}


  \begin{corollary}\label{sudden-cor}
If $z \in \F_\ZZ$ then
$\Ffpf_z \in \NN\spanning\left\{ \Ffpf_y : y \in \F_\ZZ\text{ is FPF-Grassmannian}\right\}$
and this symmetric function
is consequently Schur $P$-positive.
 \end{corollary}
 
 This leads immediately to a proof of Theorem~\ref{fpf-our-cor1} from the introduction.
 
 \begin{proof}[Proof of Theorem~\ref{fpf-our-cor1}]
 Since $\Ffpf_z$ is a Schur $P$-function if $z \in \F_\ZZ$
 is FPF-Grassmannian by Theorem~\ref{fgrass-thm},
  Corollary~\ref{sudden-cor} implies that every $\Ffpf_z$ is Schur $P$-positive.
 \end{proof}
 
 We close this section by applying Theorem~\ref{fpf-our-cor1} to compute the product of two Schur P-functions.
 Given $u \in S_m$ and $v \in S_n$, write $u \times v \in S_{m+n}$ for the permutation mapping $ i \mapsto u (i)$ for $i  \in [m]$ and $m+i \mapsto m + v(i)$ for $i \in [n]$.
 It is well known that $F_{u \times v} = F_u F_v$; for instance, this follows by applying stabilization to~\cite[Proposition 1.2]{LS}.
 An analogous result holds for FPF-involutions.
\begin{proposition}
\label{p:product}
Let $y \in \fpf_m$ and $z \in \fpf_n$.
Then $\Ffpf_{y \times z} = \Ffpf_y \Ffpf_z$.
\end{proposition}

\begin{proof}
Since $\cAfpf(y\times z) = \{ u \times v : (u,v) \in \cAfpf(y) \times \cAfpf(z)\}$, this follows from Definition~\ref{fF-def}.
\end{proof}

As a corollary, we obtain a new rule for multiplying Schur-P functions.

\begin{corollary}
\label{c:LR}
Suppose $\rho$ and $\mu$ are strict partitions. Let $y$ and $z$ be FPF-Grassmannian involutions with $\flambda(y) = \rho$ and $\flambda(z) = \mu$.
Then $P_\rho P_\mu = \sum_{\lambda} C^\lambda_{\rho \mu} P_\lambda$
where $C^\lambda_{\rho \mu}$ is the number of FPF-Grassmannian involutions with shape $\lambda$ appearing as leaves in $\fT(y\times z)$.
\end{corollary}

\begin{proof}
The result follows immediately from Proposition~\ref{p:product} and Theorem~\ref{f-finite-thm}.
\end{proof}

\begin{remark}
A similar rule can be constructed for both Schur-P and Schur-Q functions using the results in~\cite[\S4.2]{HMP4}.
\end{remark}

%
%
%
%

\section{Triangularity}\label{fpf-tri-sect}

We can show that the expansion of $\Ffpf_z$
into Schur $P$-functions is unitriangular with respect to the dominance order $\leq$ on (strict) partitions.
As in the introduction, define $\flambda(z)$ for
 $z \in \F_\infty$ to be the transpose of the partition
 given by sorting $\cfpf(z)$,
and let $\flambda(z) = \flambda(\iota(z))$ for $z \in \F_n$.

\begin{example}
Let $y = (1,8)(2,4)(3,5)(6,7) $ and $z = (1,3)(2,7)(4,8)(5,6)$ be as in as Example~\ref{fpf-dom-ex}.
Then  sorting $\cfpf(y)$ gives $(2, 1, 1, 1, 1, 1, 0, 0)$ so the shape of $y$ is $\flambda(y) = (6,1)$.
Similarly, sorting $\cfpf(z)$ gives $(2,2,1,1,0,0,0,0)$ so the shape of $z$ is $\flambda(z) = (4,2)$.
\end{example}

This construction is consistent with our earlier definition of $\flambda(z)$
when $z\in \F_\infty$ is FPF-Grassmannian.
Define $<_{\cAfpf}$ on $S_\infty$
as the transitive relation generated by setting $v <_{\cAfpf} w$
when the one-line representation of $v^{-1}$ can be transformed to that of $w^{-1}$ by replacing a
consecutive subsequence starting at an odd index
of the form $adbc$ with $a<b<c<d$ by $bcad$,
or equivalently when 
it holds for an odd number $i \in \PP$ that
\be
\label{<-a-eq}
s_iv>v > s_{i+1} v > s_{i+2}s_{i+1} v =s_{i}s_{i+1}w < s_{i+1}w < w<s_iw.
\ee
For example, $ 235164 =(412635)^{-1}  <_{\cAfpf}  (413526)^{-1} = 253146$,
but $(12534)^{-1}  \not<_{\cAfpf}  (13425)^{-1}$.
 Recall the definition of $\beta_{\min}(z)$
 from Lemma~\ref{minatomfpf-lem}.
In earlier work, we showed
\cite[Theorem 6.22]{HMP2} that $<_{\cAfpf}$ is a partial order and
that $\cAfpf(z) = \{ w \in S_\infty : \beta_{\min}(z) \leq_{\cAfpf} w\}$ for all $z \in \F_\infty$.

Write $\lambda^T$ for the transpose of a partition $\lambda$.
Then  $\lambda \leq \mu$ if and only if $\mu^T \leq \lambda^T$ \cite[Eq.\ (1.11), {\S}I.1]{Macdonald2}.
The \emph{shape} of $w \in S_\infty$
is the partition $\lambda(w)$ given by sorting $c(w)$.

\begin{lemma}\label{fpf-ut-lem1}
Let $z \in \F_\infty$. If $v,w \in \cAfpf(z)$ and $v <_{\cAfpf} w$, then
$\lambda(v) < \lambda(w)$.
\end{lemma}

\begin{proof}
Suppose $v,w \in \cAfpf(z)$ are such that 
$s_iv>v > s_{i+1} v > s_{i+2}s_{i+1} v =s_{i}s_{i+1}w < s_{i+1}w < w<s_iw$ for an odd number
$i \in \PP$,
so that $v <_{\cAfpf} w$. Define $a=w^{-1}(i+2)$, $b=w^{-1}(i)$, $c=w^{-1}(i+1)$, and $d=w^{-1}(i+3)$
so that $a<b<c<d$.
The
 diagram $D(v^{-1})$ is then given by permuting rows $i$, $i+1$, $i+2$, and $i+3$
of $D(w^{-1}) \cup \{(i+3,b),(i+3,c)\} - \{(i,a),(i+1,a)\}$,
and so $\lambda(v)$ is given by sorting $\lambda(w) - 2e_j + e_k + e_l$ for some indices
$j<k<l$ with
$\lambda(w)_j-2 \geq \lambda(w)_k \geq \lambda(w)_l$.
One checks in this case that $\lambda(v) < \lambda(w)$,
as desired.
\end{proof}

\begin{theorem}\label{fpf-ut-thm2}
Let $z \in \F_\infty$ and $\nu = \flambda(z)$.
Then $\nu^T \leq \nu$.
If $\nu^T=\nu$ then $\Ffpf_z = s_\nu$ and otherwise
$\Ffpf_z \in s_{\nu^T} + s_{\nu} +
\NN\spanning \left\{ s_\lambda : \nu^T<\lambda < \nu\right\}$.
\end{theorem}

\begin{proof}
It follows from \cite[Theorem 4.1]{Stan} that if $w \in S_\infty$ then $\lambda(w) \leq \lambda(w^{-1})^T$,
and if equality holds then $F_w = s_{\lambda(w)}$ while otherwise
$F_w \in s_{\lambda(w)} + s_{\lambda(w^{-1})^T} + \NN\spanning\{ s_\nu : \lambda(w) < \nu < \lambda(w^{-1})^T\}$.
Lemma~\ref{codef-lem} implies that $\flambda(z)^T = \lambda(\beta_{\min}(z))$,
so by Lemma~\ref{fpf-ut-lem1} we have
$\Ffpf_z = \sum_{w \in \cAfpf(z)} F_w \in s_{\flambda(z)^T} + \NN\spanning\{ s_\mu : \flambda(z)^T < \mu\}$.
The result follows since $\Ffpf_z$ is Schur $P$-positive
and each $P_\mu$ is fixed by the linear map
  $\omega : \Lambda \to \Lambda$ with $\omega(s_\mu) = s_{\mu^T}$
 for partitions $\mu$
 \cite[Example 3(a), {\S}III.8]{Macdonald2}.
 \end{proof}

%

We may finally prove Theorem~\ref{intro-fpf-tri-thm} from the introduction.

\begin{proof}[Proof of Theorem~\ref{intro-fpf-tri-thm}]
One has $P_\lambda \in s_\lambda +\NN\spanning\{ s_\nu : \nu < \lambda\}$
for any strict partition $\lambda$ \cite[Eq.\ (8.17)(ii), \S{III.8}]{Macdonald2}.
Since $\Ffpf_z$ is Schur $P$-positive, the result follows by 
Theorem~\ref{fpf-ut-thm2}.
\end{proof}

Strangely, we do not know of an easy way to show directly that $\flambda(z)$ 
is a strict partition.

\section{FPF-vexillary involutions}

Define an element
$z$ of $\F_n$ or $ \F_\ZZ$ to be \emph{FPF-vexillary} if $\Ffpf_z = P_\mu$ for a strict partition $\mu$.
In this section, we
derive a pattern avoidance condition classifying such involutions.

\begin{remark}
All FPF-Grassmannian involutions,
as well as 
all elements of $\F_n$ for $n\in\{2,4,6\}$,  are FPF-vexillary.
The sequence $(v^\fpf_{2n})_{n\geq 1} = (1, 3, 15, 92, 617,4354,\dots)$,
with $v^\fpf_n$ counting the  FPF-vexillary elements of $\F_n$,
again seems unrelated to any existing entry in \cite{OEIS}.
\end{remark}

In this section, we require the following variant of \eqref{st-eq}. For $z \in \F_\ZZ$, define
\be\label{[[]]-eq}
[[z]]_E \omdef = \ifpf([z]_E) \in \F_\infty
\ee
 for each finite set $E\subset \ZZ$
with $z(E)=E$.

\begin{lemma}\label{quasi-lem}
If $z \in \F_\ZZ$ is FPF-Grassmannian and $E\subset \ZZ$ is a finite set with $z(E) = E$, then
the fixed-point-free involution $[[z]]_E$ is also FPF-Grassmannian.
\end{lemma}

\begin{proof}
Suppose $z\in \F_\ZZ$ is FPF-Grassmannian and $E\subset \ZZ$ is finite and $z$-invariant.
We may assume that $z \in \F_\infty$ and $E\subset \PP$.
Fix a set $F = \{1,2,\dots,2n\}$ where $n \in \PP$ is large enough that $E\subset F$
and $[[z]]_F = z$.
Note that for any $z$-invariant  set $D\subset E$ we have $[[z]]_{D} = [[z']]_{D'}$ for $z'=[[z]]_E$
and $D'=\psi_E(D)$.
Inductively applying this property, we see that it suffices to show that $[[z]]_E$
is FPF-Grassmannian when $E = F\setminus\{a,b\}$
with $\{a,b\} \subset F$ a nontrivial cycle of $z$.
In this special case, it is a straightforward exercise to check that  $\Imap([[z]]_E)$
is either $ [\Imap(z)]_E$
or the involution formed by replacing the leftmost cycle of $[\Imap(z)]_E$
by two fixed points.
In either case it is easy to see that $\Imap([[z]]_E)$ is I-Grassmannian, so
$[[z]]_E$ is FPF-Grassmannian as needed.
\end{proof}

We fix the following notation in Lemmas~\ref{fpf-vex-lem1}, \ref{fpf-vex-lem2},
and \ref{fpf-vex-lem3}.
Let  $z \in \F_\ZZ-\{\wfpf\}$ 
and write $(q,r) \in \ZZ\times \ZZ$ for the maximal FPF-visible inversion of $z$.
Set $y = \fpsi(z) = (q,r)z(q,r) \in \F_\ZZ$ and define $p=y(q)<q$
so
that $\fT_1(z) = \Pfpf^-(y,p)$ if $z$ is not FPF-Grassmannian.

\begin{lemma} \label{fpf-vex-lem1}
Let $E\subset \ZZ$
be a finite set with $\{q,r\} \subset E$ and $z(E) = E$. 
Then $(\psi_E(q), \psi_E(r))$ is the maximal FPF-visible inversion of $[[z]]_E$.
Moreover, it holds that $[[\fpsi(z)]]_E = \fpsi([[z]]_E)$.
\end{lemma}

\begin{proof}
The first assertion holds since
the set of FPF-visible inversions of $z$ contained in $E\times E$
and the set of all FPF-visible inversions of $[[z]]_E$ are in bijection via the order-preserving map $\psi_E \times \psi_E$. 
The second claim follows from the definition of $\fpsi$ since $\{q,r,z(q),z(r)\} \subset E$.
\end{proof}

Define 
\be
\label{lfpf-eq}
\Lfpf(z) \omdef= \{ i \in \ZZ : i<p \text{ and }(i,p)y(i,p) \in \Pfpf^-(y,p)\}.
\ee
For any $E\subset \ZZ$ we define
\be\label{ccfpf-eq}
\Cfpf(z, E) \omdef= \left\{ (i,p)y(i,p) : i \in E\cap \Lfpf(z)\right\}.
\ee
Also let
$\Cfpf(z) \omdef= \Cfpf(z, \ZZ)$, so that $\Cfpf(z) = \fT_1(z)$
if $z$ is not FPF-Grassmannian.
The following shows that $\Cfpf(z)$ is always nonempty.

\begin{lemma} \label{fpf-vex-grass-lem}
If $z \in \F_\ZZ-\{\wfpf\}$ is FPF-Grassmannian, then $|\Cfpf(z)| = 1$. 
\end{lemma}

\begin{proof}
Assume $z \in \F_\ZZ-\{\wfpf\}$ is FPF-Grassmannian.
By Proposition~\ref{semiinverse-prop} we have $z = \Fmap(g)$ for an I-Grassmannian involution
$g \in \I_\ZZ$. Using this fact and the observations in Remark~\ref{fpf-remark},
one checks that $\Cfpf(z) = \{(i,p)y(i,p)\}$
where $i$ is the greatest integer less than $p$ such that $y(i) < q$.
\end{proof}

\begin{lemma} \label{fpf-vex-lem2}
Let $E\subset \ZZ$ be a finite set 
such that $\{q,r\} \subset E$ and $z(E) = E$.
\ben
\item[(a)] The operation $v \mapsto [[v]]_E$ restricts to an injective map $\Cfpf(z,E) \to \Cfpf([[z]]_E)$.

\item[(b)] If $E$ contains $\Lfpf(z)$, then the injective map in  (a) is a bijection.
\een
\end{lemma}

\begin{proof}
Part (a) is straightforward from the definition of $\Cfpf(z)$ given Lemma~\ref{fpf-vex-lem1}.
We prove the contrapositive of part (b).
Suppose $a<b = \psi_E(p)$ and $(a,b)[[y]]_E(a,b)$ belongs to $\Cfpf([[z]]_E)$
but is not in the image of $\Cfpf(z,E)$ under the map $v \mapsto [[v]]_E$.
Suppose $a = \psi_E(i)$ for $i \in E$.
Then $(a,b)[[y]]_E(a,b) = [[(i,p)y(i,p)]]_E$,
and 
it follows from Proposition~\ref{fpfbruhatcover-prop}
that $[[y]]_E(a) < [[y]]_E(b)$, so we likewise have $y(i) < y(p)$.
Since $(i,p)y(i,p) \notin \Cfpf(z,E)$, 
there must exist an integer $j$ with $i<j<p$ and $y(i)<y(j)<y(p)$.
Let $j$ be maximal with this property and set $k = z(j)$.
One can check using Proposition~\ref{fpfbruhatcover-prop} that 
 either $j$ or $k$ belongs to $ \Lfpf(z)$ but not $E$, so $E\not\supset \Lfpf(z)$.
\end{proof}

We say that $z \in \F_\ZZ$ \emph{contains a bad FPF-pattern} if there is a finite set $E\subset \ZZ$
with $z(E) = E$ and $|E| \leq 12$, 
such that $[[z]]_E$ is not FPF-vexillary.  We refer to  $E$
as a \emph{bad FPF-pattern} for $z$.

\begin{lemma}\label{fpf-vex-lem3}
If  $z \in \F_\ZZ$ is such that $|\fT_1(z)| \geq 2$, then $z$ contains a bad FPF-pattern.
\end{lemma}

\begin{proof}
If $u\neq v$ and $\{u,v\} \subset \fT_1(z)$,
then $u$, $v$, and $z$ agree outside a set $E\subset \ZZ$ of size 8
with $z(E)=E$.
It follows by Lemmas~\ref{fpf-vex-grass-lem} and \ref{fpf-vex-lem2}
that $E$ is a bad FPF-pattern for $z$.
\end{proof}

\begin{lemma}\label{fpf-vex-lem4}
Suppose  $z \in \F_\ZZ$ is such that $\fT_1(z) = \{v\}$ is a singleton set.
Then $z$ contains no bad FPF-patterns if and only if  $v$ contains no bad FPF-patterns.
\end{lemma}

\begin{proof}
By definition, $z$ and $v$ agree outside a set $A\subset \ZZ$ of size 6 with $v(A)=z(A)=A$.
If $z$ (respectively, $v$) contains a bad FPF-pattern that is disjoint from $A$, then 
the
other involution clearly does also.
If $z$ contains a bad FPF-pattern $B$ that intersects $A$, then $E = A\cup B$ 
has size at most 16 since $|B| \leq 12$ and both $A$ and $B$
are $z$-invariant.
In this case, 
$[[z]]_E$ contains a bad FPF-pattern
and Lemma~\ref{fpf-vex-lem2}(b)
shows that $\Cfpf([[z]]_E) = \{ [[v]]_E\}$, and if $[[v]]_E$ contains a bad FPF-pattern then $v$ does also.
By similar arguments, it follows that if 
$v$ contains a bad FPF-pattern $B$ that intersects $A$, then 
$E = A\cup B$ has size at most 16,  $[[v]]_E$ contains a bad FPF-pattern, 
$\Cfpf([[z]]_E) = \{ [[v]]_E\}$, and $v$ contains a bad FPF-pattern if $[[v]]_E$ does.

These observations
show that to prove the lemma, it suffices to consider the case when $z$ 
belongs to the image of $\ifpf : \F_{16} \hookrightarrow \F_\ZZ$.
Using a computer, we have checked that if $z$ is such an involution and 
 $\Cfpf(z) = \{v\}$ is a singleton set, then
$z$ contains no bad FPF-patterns if and only if  $v$ contains no bad FPF-patterns. 
There are  940,482 
possibilities for $z$, a sizeable but tractable number.
\end{proof}

\begin{theorem}\label{fpf-vex-thm}
 An involution $z \in \F_\ZZ$
is FPF-vexillary if and only if $[[z]]_E$ is FPF-vexillary for all sets $E\subset \ZZ$ with $z(E) = E$
and $|E| =12$.
\end{theorem}

\begin{proof}
Let $\cX \subset \F_\ZZ$ be the set that contains $z\in \F_\ZZ$
 if and only if $z$ is FPF-Grassmannian or 
$\fT_1(z) = \{v\}$ and $v \in \cX$.
It follows from Corollary~\ref{ft-cor1}(b) that 
$\cX$ is the set of all FPF-vexillary involutions in $\F_\ZZ$.
On the other hand,
Lemmas~\ref{quasi-lem}, \ref{fpf-vex-lem3}, and \ref{fpf-vex-lem4}
show that $\cX$ is the set of involutions $z \in \F_\ZZ$
that contain no bad FPF-patterns.
Thus $z \in \F_\ZZ$ is FPF-vexillary if and only if $z$ has no
bad FPF-patterns, which is equivalent to the theorem statement.
\end{proof}

\begin{corollary}\label{fvex-cor}
An involution $z \in \F_\ZZ$ is FPF-vexillary if and only if for
all finite sets $E\subset \ZZ$ with $z(E) = E$
the involution $[z]_E$ is not any of the following sixteen permutations:
{\small
\[
\ba
(1,3)(2,4)(5,8)(6,7),        &&&&&&& (1,5)(2,3)(4,7)(6,8),   &&&&&&    (1,6)(2,4)(3,8)(5,7), \\
(1,3)(2,5)(4,7)(6,8),        &&&&&&& (1,5)(2,3)(4,8)(6,7),   &&&&&&    (1,6)(2,5)(3,8)(4,7), \\
(1,3)(2,5)(4,8)(6,7),        &&&&&&& (1,5)(2,4)(3,7)(6,8),   &&&&&&    (1,3)(2,4)(5,7)(6,9)(8,10), \\
(1,3)(2,6)(4,8)(5,7),        &&&&&&& (1,5)(2,4)(3,8)(6,7),   &&&&&&    (1,3)(2,5)(4,6)(7,9)(8,10), \\
(1,4)(2,3)(5,7)(6,8),        &&&&&&& (1,6)(2,3)(4,8)(5,7),   &&&&&&    (1,3)(2,4)(5,7)(6,8)(9,11)(10,12). \\
(1,4)(2,3)(5,8)(6,7),
\ea
\]
}
\end{corollary}

\begin{proof}
It follows by a computer calculation using the formulas
in Theorems \ref{fgrass-thm} and \ref{f-finite-thm} that
$z \in \ifpf(\F_{12}) \subset \F_\infty$ is not FPF-vexillary if and only if
there is a $z$-invariant subset $E\subset \ZZ$ such that $[z]_E$
is one of the given involutions. The corollary follows from this fact by Theorem \ref{fpf-vex-thm}.
\end{proof}

\section{Pfaffian formulas}\label{pfaffian-fpf-sect}

The \emph{Pfaffian} of a skew-symmetric $n\times n$ matrix $A$
is 
\be\label{pf-eq}
 \pf A
\omdef= \sum_{z \in \F_n} (-1)^{\ellfpf(z)} \prod_{z(i)<i \in [n] } A_{z(i),i}
.\ee
It is a classical fact that $\det A  = (\pf A)^2$. Since $\det A = 0$ when $A$ is skew-symmetric
but $n$ is odd, the definition \eqref{pf-eq}
is consistent with the fact that the set $\F_n$ of fixed-point-free involutions in
$S_n$ is nonempty only if $n$ is even.
If $A = (a_{ij})$ is a $2\times 2$ skew-symmetric matrix then $\pf A = a_{12} = -a_{21}$.
If $A = (a_{ij})$ is a $4\times 4$ skew-symmetric matrix then
$\pf A = a_{21} a_{43} - a_{31}a_{42} + a_{41}a_{32}$.

Both  $\Sfpf_z$ and $\Ffpf_z$ can be expressed by certain Pfaffian formulas 
when $z$ is FPF-Grassmannian.
We fix the following notation for the duration of this section: first, let
\be
\label{fpf-nrphi-eq}
n,r \in \PP \qquand \phi \in \PP^r \text{ with }0 < \phi_1 < \phi_2<\dots <\phi_r< n.
\ee
Set $\phi_i=0$ for $i>r$.
Define
$y =  (\phi_1, n+1)(\phi_2,n+2)\cdots (\phi_r,n+r) \in \I_\infty$ and
$z= \Fmap(y)$.
Let
\be\label{sfpf-ffpf-eq}
\Sfpf[\phi_1,\phi_2,\dots,\phi_r;n]  \omdef= \Sfpf_z
\qquand
\Ffpf[\phi_1,\phi_2,\dots,\phi_r;n] \omdef= \Ffpf_z.
\ee
When $r$ is odd, we also set $\Sfpf[\phi_1,\phi_2,\dots,\phi_r,0;n] \omdef= \Sfpf_z$ and
$\Ffpf[\phi_1,\phi_2,\dots,\phi_r,0;n] \omdef= \Ffpf_z$.

\begin{proposition}\label{fpf-pf-prop}
In the notation just given, $z \in \F_\infty$ is FPF-Grassmannian
with shape $\flambda(z) = (n-\phi_1, n-\phi_2,\dots,n-\phi_r)$.
Moreover, each FPF-Grassmannian element of $\F_\infty-\{\wfpf\}$
occurs as such an involution $z$ for a unique choice of
$n,r \in \PP$ and $\phi \in \PP^r$ as in \eqref{fpf-nrphi-eq}.
\end{proposition}

\begin{proof}
Let $X = [n]\setminus \{\phi_1,\phi_2,\dots,\phi_r\}$ so that $n \in X$.
If $|X|$ is even then $\Imap(z) = y$. If $|X|$ is odd and at least 3, then $\Imap(z) = y \cdot (n,n+r+1)$.
If $|X|= 1$, finally, then $\phi = (1,2,\dots,n-1)$ and $\Imap(z) = (2,n+2)(3,n+3)\cdots(n,2n)$.
In each case, 
$\flambda(z) =  (n-\phi_1, n-\phi_2,\dots,n-\phi_r)$ as desired.
The second assertion holds
since
an FPF-Grassmannian element of $\F_\infty$ is uniquely determined by
its image under $\Imap : \F_\infty \to \I_\infty$,
which must be I-Grassmannian with  an even number of fixed points in $[n]$
and
not equal to $(i+1,n+1)(i+2,n+2)\cdots(n,2n-i)$ for any $i \in [n]$. 
\end{proof}

Let $\ell^+(\phi)$ be whichever of $r$ or $r+1$ is even,
and let $[a_{ij}]_{1\leq i < j \leq n}$ denote the skew-symmetric matrix with $a_{ij}$
in position $(i,j)$ and $-a_{ij}$ in position $(j,i)$ for $i<j$ (and zeros on the diagonal).

\begin{corollary}\label{fpf-jt-cor}
In the setup of \eqref{fpf-nrphi-eq},
 $\Ffpf[\phi_1,\phi_2,\dots,\phi_r;n]
 = \pf\left[ \Ffpf[\phi_i,\phi_j;n]\right]_{1\leq i < j\leq \ell^+(\phi)}.$
\end{corollary}

\begin{proof}
Macdonald proves that 
if $\lambda$ is a strict partition then
$P_\lambda = \pf[P_{\lambda_i\lambda_j}]_{1\leq i <j \leq \ell^+(\lambda)}$
 \cite[Eq.\ (8.11), {\S}III.8]{Macdonald2}.
 Given this fact
 and the preceding proposition,
the result follows from  Theorem~\ref{fgrass-thm}.
\end{proof}

Our goal is to prove that the identity in this corollary
holds 
with $\Ffpf[{\cdots};n]$ replaced by $\Sfpf[{\cdots};n]$.
In the following lemmas, we let
\be
\label{mfpf-eq}
\Mfpf[\phi;n]  =\Mfpf[\phi_1,\phi_2,\dots,\phi_r;n] \omdef= \left[\Sfpf[\phi_i,\phi_j;n]\right]_{1\leq i <j \leq \ell^+(\phi)}
\ee denote the $\ell^+(\phi)\times \ell^+(\phi)$ skew-symmetric matrix with $\Sfpf[\phi_i,\phi_j;n]$ is position $(i,j)$ for $i<j$.

\begin{lemma}\label{fpf-pf-lem1}
Maintain the notation of \eqref{fpf-nrphi-eq}, and suppose $p \in [n-1]$. Then
\[
\partial_p  \(\pf \Mfpf[\phi;n] \)=
\begin{cases}
\pf \Mfpf[\phi+e_i;n] &\text{if }p = \phi_i \notin \{ \phi_2-1, \dots,\phi_r-1\}\text{ for some $i \in [r]$}\\
0&\text{otherwise}
\end{cases}
\]
where $e_i = (0,\dots,0,1,0,0,\dots)$ is the standard basis vector
whose $i$th coordinate is 1.
\end{lemma}

\begin{proof}
Let $\fk M = \Mfpf[\phi;n]$. If $1\leq i < j \leq \ell^+(\phi)$ then \eqref{fpf-eq} implies that
$\partial_p\fk M_{ij} =\partial_p \Sfpf[\phi_i,\phi_j;n]$ is
$\Sfpf[\phi_i +1, \phi_j]$ if $p = \phi_i  \neq \phi_j-1$,
$\Sfpf[\phi_i, \phi_j+1]$ if $p=\phi_j$,
and 0 otherwise.
Thus if $p \notin \{\phi_1,\phi_2,\dots,\phi_r\}$ then $\partial_p  \(\pf \fk M \) = 0$.
Suppose $p = \phi_k$.
Then $\partial_p\fk M_{ij}=0$ unless $i=k$ or $j=k$,
so 
$\partial_p  \(\pf \fk M \) = \pf \fk N$ where $\fk N$
is the matrix formed by applying $\partial_p$ to the entries in the $k$th row and $k$th column
of $\fk M$.
If $k<r$ and $\phi_k= \phi_{k+1}-1$, then columns $k$ and $k+1$ of $\fk N$ are identical,
so   $\pf \fk M= \pf \fk N = 0$.
If $k=r$ or if $k<r$ and $\phi_k \neq \phi_{k+1}-1$, then $\fk N = \Mfpf[\phi+e_k;n]$.
\end{proof}

\begin{lemma} \label{fpf-pf-lem2}
Let $n\geq 2$ and $D=(x_1+x_2)(x_1+x_3)\cdots (x_1+x_{n})$.
Then
$\pf \Mfpf[1;n] = D$, and if $b\in \PP$ is such that $1<b< n$, then $\pf \Mfpf[1,b;n]$
is divisible by $D$.
\end{lemma}

\begin{proof}
Theorem \ref{fpfdominant-thm} implies 
that $\pf \Mfpf[1;n] = D$
and, when $n> 2$, that
 $\pf \Mfpf[1,2;n] = (x_2+x_3)\cdots (x_2+x_n)D$.
 If $2<b<n$
 then 
 $ \pf \Mfpf[1,b;n] = \partial_{b-1} ( \pf \Mfpf[1,b-1;n])$
 by the previous lemma.
 Since $D$ is symmetric in $x_{b-1}$ and $x_b$,
the desired property holds by induction.
\end{proof}

If $i : \PP\to \NN$ is a map with $i^{-1}(\PP)\subset [n]$,
then we define $x^i = x_1^{i(1)} x_2^{i(2)}\cdots x_n^{i(n)}$.
Given a nonzero polynomial
$f = \sum_{i : \PP \to \NN} c_i x^i \in \ZZ[x_1,x_2,\dots]$,
let $j : \PP \to \NN$
be the lexicographically minimal index
such that $c_j \neq 0$ and define $\minlex (f) = c_j x^j$.
We refer to $\minlex(f)$ as the
\emph{least term} of $f$.
Set $\minlex(0) = 0$,
so that $\minlex(fg) = \minlex(f) \minlex(g)$ for any polynomials $f,g$.
The following is \cite[Proposition 3.14]{HMP1}.

\begin{lemma}[See \cite{HMP1}]\label{fpf-minlex-lem}
If $z \in \F_\infty$ then
$\minlex (\Sfpf_z) =  x^{\cfpf(z)}= \prod_{(i,j) \in \Dfpf(z)} x_i $.
\end{lemma}

Let $\mathscr{M}$ denote the set of monomials $x^i = x_1^{i(1)}x_2^{i(2)}\cdots$
for maps $i : \PP \to \NN$
with $i^{-1}(\PP)$ finite.
Define $\prec$ as the ``lexicographic''  order on $\mathscr{M}$,
that is, the order with $x^i \prec x^j$
when there exists $n \in \PP$ such that
 $i(t) = j(t)$ for $1 \leq t < n$ and $i(n) < j(n)$.
 Note that $\minlex(\Sfpf_z) \in \mathscr{M}$.
  Also, observe that if $a,b,c,d \in \mathscr{M}$ and $a\preceq c$ and $b\preceq d$, then
  $ab \preceq cd$ with equality if and only if $a=c$ and $b=d$.

\begin{lemma}\label{fpf-pf-igrass-lem}
Let $i,j,n \in \PP$. The following identities then hold:
\ben
\item[(a)] If $i < n$ then $\minlex ( \Sfpf[i;n] ) \succeq x_{i+1}x_{i+2}\cdots x_{n}$,
with equality if and only if $i$ is odd.
\item[(b)]
If $i<j< n$ then $\minlex ( \Sfpf[i,j;n] ) \succeq (x_{i+1}x_{i+2}\cdots x_n)(x_{j+1}x_{j+2}\cdots x_n)$,
with equality if and only if $i$ is odd and $j$ is even.
\een
\end{lemma}

\begin{proof}
The result follows by routine calculations using Lemma~\ref{fpf-minlex-lem}.
For example, suppose $i<j<n$ and let $y = (i,n+1)(j,n+2)$ and  $z = \Fmap(y)$,
so that $\Sfpf[i,j;n] =\Sfpf_z$.
If $i$ is even and $j=i+1$,
 then $\Dfpf(z) = \{(i,i-1),(i+1,i-1)\} \cup \{(i+1,i), (i+3,i),\dots,(n,i)\}\cup \{(i+3,i+1),\dots,(n,i+1)\}$
  so $\minlex( \Sfpf[i,j;n]) = (x_{i}x_{i+1}x_{i+3}\cdots x_n)(x_{j}x_{j+2}\cdots x_n)$.
 The other cases follow by similar analysis.
\end{proof}

\begin{lemma}\label{fpf-pf-lem3}
If  $n \in \PP$ and $r \in [n-1]$ then 
$\Sfpf[1,2,\dots,r;n] = \pf \Mfpf[1,2,\dots,r;n].$
\end{lemma}

\begin{proof}
The proof is similar to that of \cite[Lemma 4.77]{HMP4}.
Let
$D_i = (x_i+x_{i+1})(x_i+x_{i+2})\cdots(x_i+x_n)$ for $i \in [n-1]$
and  $\fk M = \Mfpf[1,2,\dots,r;n]$.
Theorem \ref{fpfdominant-thm} implies that
$\Sfpf[1,2,\dots,r;n]  = D_1D_2\cdots D_r$. 
Lemma~\ref{fpf-pf-lem1} 
implies that $\pf \fk M$ is symmetric in $x_1,x_2,\dots,x_r$.
Lemma~\ref{fpf-pf-lem2}
implies that every entry in the first column of $\fk M$,
and therefore also $\pf \fk M$, is divisible by $D_1$.
Since $s_i(D_i)$ is divisible by $D_{i+1}$,
it follows that 
$\pf \fk M$ is 
 divisible by $\Sfpf[1,2,\dots,r;n]$.
 To prove the lemma, it suffices to show that
$\pf \fk M$ and $\Sfpf[1,2,\dots,r;n]$ have the same least term.

Let $m \in \PP$ be whichever of $r$ or $r+1$ is even
and choose $z \in \F_{m}$.
By
 Lemma \ref{fpf-pf-igrass-lem},
\[
\minlex \(\prod_{z(i)<i \in [m] } \fk M_{z(i),i}\) 
\succeq (x_2\cdots x_n)(x_3\cdots x_n)\cdots (x_{r+1}\cdots x_n) = \minlex( \Sfpf[1,2,\dots,r;n]) 
,\] with equality if and only if $i$ is odd and $j$ is even whenever $i < j = z(i)$.
The only element  $z \in \F_m$ with the latter property is the involution 
$z  = (1,2)(3,4)\cdots(m-1,m) = \wfpf_m$, so
we deduce from  \eqref{pf-eq}  
 that $\minlex( \pf \fk M) = \minlex( \Sfpf[1,2,\dots,r;n])$ as needed.
\end{proof}

Let $\Sfpf[\phi;n] = \Sfpf[\phi_1,\phi_2,\dots,\phi_r;n]$.
The following is the main result of this section.

\begin{theorem}\label{tt:pfaffian}
It holds that $\Sfpf[\phi;n] = \pf \Mfpf[\phi;n]$.
\end{theorem}

\begin{proof}
If $\phi = (1,2,\dots,r)$ then $\Sfpf[\phi;n] = \pf \Mfpf[\phi;n]$ by the previous lemma.
Otherwise, there exists a smallest  $i \in [r]$ such that $i<\phi_i$. If $p = \phi_i-1$
then $\Sfpf[\phi;n] = \partial_p \Sfpf[\phi-e_i;n]$ 
by \eqref{fpf-eq}  and $\pf \Mfpf[\phi;n] = \partial_p( \pf \Mfpf[\phi-e_i;n])$
by Lemma~\ref{fpf-pf-lem1}. 
We may assume that $\Sfpf[\phi-e_i;n] = \pf \Mfpf[\phi-e_i;n]$ by induction, so the result follows.
\end{proof}

\begin{example}
\label{e:pfaffian}
For $\phi = (1,2,3)$ and $n=4$ the theorem reduces to the identity
\[
\Sfpf_{(1,5)(2,6)(3,7)(4,8)}
=
\pf \(\barr{rrrr}
0 & \Sfpf_{(1,5)(2,6)(3,4)} & \Sfpf_{(1,5)(2,4)(3,6)} & \Sfpf_{(1,5)(2,3)(4,6)} \\
-\Sfpf_{(1,5)(2,6)(3,4)} & 0 & \Sfpf_{(1,4)(2,5)(3,6)} & \Sfpf_{(1,3)(2,5)(4,6)} \\
-\Sfpf_{(1,5)(3,6)(2,4)} & -\Sfpf_{(1,4)(2,5)(3,6)} & 0 & \Sfpf_{(1,2)(3,5)(4,6)} \\
-\Sfpf_{(1,5)(2,3)(4,6)} & -\Sfpf_{(1,3)(2,5)(4,6)} & -\Sfpf_{(1,2)(3,5)(4,6)} & 0
\earr\)
\]
where for $z \in \F_n$ we define $\Sfpf_z = \Sfpf_{\ifpf(z)}$.
By Theorem \ref{fpfdominant-thm}, both of these expressions evaluate to
$(x_1+x_2)(x_1+x_3)(x_1+x_4)(x_2+x_3)(x_2+x_4)(x_3+x_4)$.
\end{example}

\newpage
\appendix
\section{Index of symbols}
\label{not-sect}

The tables below list our non-standard notations, with references to definitions where relevant.

\def\skip{\\[-4pt]}

\begin{center}
\begin{tabular}{l | l r}
\text{Symbol} & \text{Meaning} &\text{Reference}
 \\
 \hline
 $\NN$ & The set of nonnegative integers \\
$\PP$ & The set of positive integers \\
$[n]$ & The set of positive integers $\{1,2,\dots,n\}$ \\
$\phi_E $ & The unique order-preserving bijection $ [n] \to E$ for $E \subset \ZZ$ \\
$\psi_E $ & The unique order-preserving bijection $ E \to [n]$ for $E \subset \ZZ$ \\
\skip
$S_\ZZ$ & The group of permutations of $\ZZ$ with finite support \\
$\I_\ZZ$ & The set $\{ w \in S_\ZZ : w=w^{-1}\}$ of involutions in $S_\ZZ$ \\
 $S_\infty$ & Subgroup of permutations in $S_\ZZ$ fixing all numbers outside $\PP$ & \\  
$\I_\infty$ & The set $\{ w \in S_\infty : w=w^{-1}\}$  of involutions in $S_\infty$ \\
  $S_n$ & Subgroup of permutations in $S_\infty$ fixed all numbers outside $[n]$ & \\  
  \skip
$\wfpf$ & The permutation of $\ZZ$ given by $i \mapsto i - (-1)^i$ & \eqref{wfpf-eq} \\
$\wfpf_n$  & The permutation  $(1,2)(3,4) \dots (2n-1,2n) \in S_{2n}$ \\
$\F_n$ & The set of fixed-point-free involutions in $S_{2n}$ \\
$\F_\infty$ & The $S_\infty$-conjugacy class of $\wfpf$  & \S\ref{fpf-schub-sect} \\
$\F_\ZZ$ & The $S_\ZZ$-conjugacy class of $\wfpf$  & \S\ref{fpf-schub-sect} \\
\skip
$\iota$ & The natural inclusion $\F_n \hookrightarrow \F_\infty$ & \eqref{ifpf-def} \\
$\Fmap$ & A certain map $\I_\ZZ \to \F_\ZZ$ & Def.~\ref{Fmap-def} \\
$\Imap$ & A certain map $\F_\ZZ \to \I_\ZZ$ & Def.~\ref{I-def} \\
$\fpsi$ & A certain map $\F_\ZZ-\{\wfpf\} \to \F_\ZZ$ & Def.~\ref{fpsi-def} \\
\skip
$w_n$ & The longest permutation $n \cdots 321 \in S_n$  & \\
$[w]_E$ & The standardization of $w$ to the subset $E\subset \ZZ$ & \eqref{st-eq} \\
$[[w]]_E$ & The element $\iota([w]_E) \in \F_\infty$ for $E\subset \ZZ$ with $w(E)=E$  & \eqref{[[]]-eq}\\
$w \gg N$ & The map $\ZZ \to \ZZ$ given by $i \mapsto w(i-N) + N$ &  \\
\skip
$\cR(w)$ & The set of reduced words for $w \in W$ & \S\ref{prelim-sect} \\
$\cAfpf(z)$ & The set of minimal length elements $w \in S_\ZZ$  with $z=w^{-1} \wfpf w$  & \\
$\cRfpf(z)$ & The disjoint union $\cRfpf(z) = \bigsqcup_{w \in \cAfpf(z)} \cR(w)$ & \eqref{crfpf-eq} \\
$\beta_{\min}(z)$ & The minimal atom in $\cAfpf(z)$ for $z \in \F_\infty$ & Lem.~\ref{minatomfpf-lem} \\
\skip
$\Cyc_\ZZ(z)$ & The set $\{(i,j) \in \ZZ \times \ZZ: i < j = z(i)\}$ for $z \in \F_\ZZ$ & \eqref{inv-fpf-eq} \\
$\Cyc_\PP(z)$ & The intersection $\Cyc_\ZZ(z) \cap (\PP \times \PP)$ \\
$\inv(z)$  & The inversion set $\{ (i,j) \in \ZZ\times \ZZ : i< j \text{ and }z(i)>z(j)\}$ &  \\
$\inv_\fpf(z)$ & The set $\inv(z) - \Cyc_\ZZ(z)$ for $z\in \F_\ZZ$  & \eqref{inv-fpf-eq} \\
\skip
$\ellfpf$ & The FPF-involution length function $\F_\ZZ \to \NN$ & \eqref{ell-fpf-eq} \\
$\DesF(z)$ & A modified right descent set for $z \in \F_\ZZ$ & \eqref{ell-fpf-eq} \\
$\DesIF(z)$ & The set of FPF-visible descents of $z \in \F_\ZZ$ & \eqref{desif-eq} \\
$\DesI(z)$ & The set of visible descents of $z \in \I_\ZZ$ & \eqref{desi-eq} \\
\end{tabular}
\end{center}

\begin{center}
\begin{tabular}{l | l r}
\text{Symbol} & \text{Meaning} &\text{Reference}
 \\
 \hline
$\fkS_w$ & The Schubert polynomial of $w \in S_n$ & \eqref{maineq} \\
$\Sfpf_z$ & The FPF-involution Schubert polynomial $\sum_{w \in \cAfpf(z)} \fkS_w$ & Def.~\ref{fS-def} \\
$F_w$ & The Stanley symmetric function of $w \in S_n$ & Def.~\ref{F-def} \\
$\Ffpf_z$ & The FPF-involution symmetric function $\sum_{w \in \cAfpf(z)} F_w$ & Def.~\ref{fF-def} \\
\skip
  $<$ & The Bruhat order on $S_\ZZ$ or $\F_\ZZ$ & \S\ref{invtrans-sect} \\
  $\lessdot_\F$ & The covering relation for the Bruhat order on $\F_\ZZ$  & \S\ref{invtrans-sect} \\
  $<_{\cAfpf}$ & A certain partial order of $\cAfpf(z)$ & \eqref{<-a-eq} \\
\skip
$D(w)$ & The Rothe diagram $\{ (i, w(j)) : (i,j) \in \inv(w)\}$ & \eqref{d-eq} \\
$\Dfpf(z)$ & The involution Rothe diagram of $z \in \F_\infty$ & \eqref{dfpf-eq} \\
$c(w)$ & The code of $w \in S_\infty$  & \S\ref{fgrass-sect} \\
$\cfpf(z)$ & The involution code of $w \in \F_\infty$ & \S\ref{fgrass-sect} \\
\skip
$\lambda(w)$ & The partition given by sorting $c(w)$ for $w \in S_\infty$ & \S\ref{fpf-tri-sect} \\
$\flambda(z)$ & The shape of $w \in \F_\infty$ & \S\ref{fpf-tri-sect} \\
$\delta_n$ & The partition $(n-1,n-2,\dots,3,2,1)$ \\
$\lambda^T$ & The transpose of a partition $\lambda$ \\
\skip
$\cP$ & The polynomial ring $\ZZ\left[x_1,x_2,\dots\right]$  \\
$\cL$ & The Laurent polynomial ring $\ZZ\left[x_1,x_2,\dots,x_1^{-1},x_2^{-1},\dots\right]$  \\
$\partial_i$ & The $i$th divided difference operator & \eqref{partial-i-eq} \\
$\pi_i$ & The $i$th isobaric divided difference operator &\eqref{pi-i-eq} \\
$G_{m,n}$ & A certain element of $\cL$ & \eqref{Gmn-eq} \\
\skip
$\Lambda$ & The Hopf algebra of symmetric functions over $\ZZ$ & \cite{EC2} \\
$s_\lambda$ & The Schur function indexed by a partition $\lambda$ & \cite{EC2} \\
$P_\lambda$ & The Schur $P$-function indexed by a strict partition $\lambda$ & Def.~\ref{schurp-def} \\
\skip
$\Pfpf^\pm(y,r)$ & Index sets for sums in transition formula Theorem~\ref{monkfpf-thm} & \eqref{Pfpf-def} \\
$\fT(z) $ & The FPF-involution Lascoux-Sch\"utzenberger tree & Def.~\ref{fT-def}  \\
$\Lfpf(z) $ & The set $\{ i \in \ZZ : i<p \text{ and }(i,p)y(i,p) \in \Pfpf^-(y,p)\}$ & \eqref{lfpf-eq}  \\
$\Cfpf(z,E) $ & The set $ \left\{ (i,p)y(i,p) : i \in E\cap \Lfpf(z)\right\}$ & \eqref{ccfpf-eq}  \\
\skip
$\pf A$ & The Pfaffian of a skew-symmetric matrix $A$ & \eqref{pf-eq} \\
$\Sfpf[\phi;n] $ & An instance of $\Sfpf_z$ where $z$ is FPF-Grassmannian & \eqref{sfpf-ffpf-eq} \\
$\Ffpf[\phi;n]$ & An instance of $\Ffpf_z$ where $z$ is FPF-Grassmannian & \eqref{sfpf-ffpf-eq} \\
$\Mfpf[\phi;n]$ & A certain skew-symmetric matrix & \eqref{mfpf-eq}
\end{tabular}
\end{center}

\end{document}